\pdfoutput=1

\documentclass[11pt]{amsart}

\usepackage{amsfonts,amssymb,verbatim,amsmath,amsthm,latexsym,textcomp,amscd}
\usepackage{latexsym,amsfonts,amssymb,epsfig,verbatim}
\usepackage{amsmath,amsthm,amssymb,latexsym,graphics,textcomp,hyperref}
\usepackage{graphicx}
\usepackage{color}
\usepackage{url}
\usepackage{psfrag}
\usepackage{amsmath,comment,amscd}
\usepackage{amssymb}
\usepackage{amsthm}

\usepackage{caption}
\usepackage{color}
\usepackage{enumerate}
\usepackage{float}
\usepackage[a4paper, hmargin=2.7cm, vmargin=2.3cm]{geometry}
\usepackage{graphicx}
\usepackage{tikz-cd}

\numberwithin{equation}{section}

\newcommand{\define}{\textbf}

\newcommand{\mc}{\mathcal}

\newcommand{\ra}{\Rightarrow}

\newcommand{\vp}{\varphi}

\newcommand{\ab}[1]{\langle#1\rangle} 

\newcommand{\lrb}[1]{\left[#1\right]}

\newcommand{\lrp}[1]{\left(#1\right)}

\newcommand{\norm}[1]{\|#1\|}
\newcommand{\set}[1]{\{#1\}}
\newcommand{\suppressthis}[1]{}

\newcommand{\reason}[2]{\stackrel{\text{#1}}{#2}}

\newcommand{\R}{\mathbb R}
\newcommand{\N}{\mathbb N}
\newcommand{\Q}{\mathbb Q}
\newcommand{\T}{\mathbb T}
\newcommand{\Z}{\mathbb Z}
\newcommand{\C}{\mathbb C}

\newtheorem{theorem}{Theorem}[section]
\newtheorem{lemma}[theorem]{Lemma}
\newtheorem{corollary}[theorem]{Corollary}

\theoremstyle{definition}
\newtheorem{definition}{Definition}[section]
\newtheorem{remark}[theorem]{Remark}

\theoremstyle{definition}
\newtheorem{conjecture}{Conjecture}

\DeclareMathOperator{\supp}{supp}

\newcommand{\one}{\mathbf 1}

\begin{document}

\title{A Periodicity Result for Tilings of $\Z^3$ by Clusters of Prime-Squared Cardinality}
\author{Abhishek Khetan}
\address{Visiting Professor, Department of Computer Science, Ashoka University, Sonipat, Haryana, India}

\email{abhishek.khetan@ashoka.edu.in}



\date{\today}

\begin{abstract}
	We show that if $\Z^3$ can be tiled by translated copies of a set $F\subseteq\Z^3$ of cardinality the square of a prime then there is a weakly periodic $F$-tiling of $\Z^3$, that is, there is a tiling $T$ of $\Z^3$ by translates of $F$ such that $T$ can be partitioned into finitely many $1$-periodic sets.
\end{abstract}

\maketitle

\tableofcontents


\section{Introduction}

Let $F$ be a finite subset of $\Z^n$, which we will refer to as a \emph{cluster}.
A \emph{tiling} of $\Z^n$ by translates of $F$, or an $F$-\emph{tiling}, is, roughly, a covering of $\Z^n$ by non-overlapping translates of $F$.
The \emph{periodic tiling conjecture} (See \cite{lagarias_wang_tiling}) states that if there is an $F$-tiling of $\Z^n$ for some cluster $F$, then there is an $F$-tiling which is invariant under translations by a finite index subgroup of $\Z^n$, that is, there is a \emph{fully periodic} $F$-tiling.
It was proved in \cite{beauquier_nivat_1991} that if $F$ is a \emph{connected} cluster (which roughly means that there are no `gaps' in $F$) in $\Z^2$ then every $F$-tiling is periodic in at least one direction.
This is sufficient to establish the existence of a fully periodic tiling by a simple pigeon-hole argument.
Important progress towards the periodic tiling conjecture was made in \cite{szegedy_algorithms_to_tile} in which it was proved that the periodic tiling conjecture holds (in any dimension) for clusters of prime cardinality.
Further major progress was made in \cite{bhattacharya_tilings} where it was established that the periodic tiling conjecture holds in $\Z^2$, without any constraint on the cardinality or the geometry of the cluster.
The main result of \cite{bhattacharya_tilings} shows that the \emph{orbit closure} of any $F$-tiling, where $F$ is a cluster in $\Z^2$, has a tiling which satisfies a certain weak notion of periodicity, which was shown to be sufficient to guarantee the existence of a fully periodic tiling.
To make this precise, first we define a subset $S$ of $\Z^2$ as \emph{$1$-periodic} if there is a nonzero vector $g$ in $\Z^2$ such that $g+S=S$.
We say that a subset $S$ of $\Z^2$ is \emph{weakly periodic} if $S$ can be partitioned in to finitely many $1$-periodic subsets.
Bhattacharya \cite{bhattacharya_tilings} shows that the orbit closure of any $F$-tiling has a weakly periodic tiling in it.
A stronger version of this result was obtained in \cite{greenfeld_tao_2020} which says that \emph{every} $F$-tiling of a cluster in $\Z^2$ is weakly periodic, thereby removing the need to pass to the orbit closure.

It was recently shown in \cite{greenfel_tao_counterexample_ptc} that the periodic tiling conjecture is false in $\Z^n$ if $n$ is sufficiently large.
However, we conjecture that the following weaker statement might still hold.

\begin{conjecture}
	\textit{\textbf{Weakly Periodic Tiling Conjecture.}}
	Let $F\subseteq\Z^n$ be an exact cluster.
	Then there is a weakly periodic $F$-tiling.
	More precisely, there exists an $F$-tiling $T\subseteq\Z^n$ such that there is a finite partition $T = T_1\sqcup \cdots \sqcup T_k$ such that each $T_i$ is $1$-periodic.
\end{conjecture}
In this paper we show that if $F\subseteq \Z^3$ is a cluster of cardinality the square of a prime such that there exists an $F$-tiling of $\Z^3$, then there exists a weakly periodic $F$-tiling of $\Z^3$ (See Theorem \ref{theorem:prime square theorem}).
This makes some progress towards the above conjecture beyond the already known result of Szegedy \cite{szegedy_algorithms_to_tile}.
Unlike in two dimensions, the existence of a weakly periodic tiling alone does not seem (to the author) to imply the existence of a fully periodic tiling.
In fact, to upgrade from a $1$-periodic tiling of $\Z^3$ to a fully-periodic $F$-tiling also seems hard.
The difficulty is genuine, since the problem of this upgradation is a close cousin of the periodic tiling conjecture for the group $\Z^2\times (\Z/N\Z)$ which, as noted in \cite{greenfeld_tao_undecidable_trans}, is as yet an unresolved problem.


%
%
%

%
\subsection{Overview of the Proof}

In \cite{bhattacharya_tilings} the problem of finding a fully periodic $F$-tiling for a cluster $F\subseteq \Z^2$ was first reduced to showing the existence of a weakly periodic $F$-tiling by an elementary combinatorial argument.
The existence of a weakly periodic $F$-tiling was then shown to follow from the following dynamical statement: If $\Z^2$ acts ergodically on a probability space $(X, \mu)$ and $A$ is a subset of $X$ such that finitely many $\Z^2$-translates of $A$ partition $X$, then $A$ itself is weakly periodic.
To approach this problem, Bhattacharya uses the spectral theorem to transfer the problem to $L^2(\T^2, \nu)$, where $\nu$ is the spectral measure associated with the characteristic function of $A$.
Then using a dilation lemma and an averaging argument, it was shown that the spectral measure is supported on finitely many $1$-dimensional affine subtori of $\T^2$.
Then the averages of $1_A$ along the `directions' of these subtori (when $\T^2$ is thought of as $\R^2/\Z^2$, each $1$-dimensional subtorus can be thought of as a `line' and hence has a direction) were shown to behave polynomially in $\R/\Z$.
Since a polynomial in $\R/\Z$ is either periodic or equidistributes, the space $(X, \mu)$ gets partitioned into ergodic components of a finite index subgroup of $\Z^2$, such that on each ergodic component, the averages either `equidistribute' or are constant.
The rest of the proof goes by carefully dealing with the equidistribution case.

Thus a key step in the argument was to show that the spectral measure is supported on a `thin' subset of $\T^2$.
We take this as a cue and study the problem for a cluster $F\subseteq \Z^3$.
If the corresponding spectral measure again happens to be supported on finitely many lines then Bhattacharya's proof goes through \emph{mutatis mutandis}.
In the adverse case the spectral measure could be supported on `planes.'
To deal with these adverse cases we use the hypothesis that the size of the cluster is the square of a prime.
This is done by Lemma \ref{lemma:spectral support lemma used in case two} which places a geometric constraint on a cluster of prime-power cardinality in the adverse cases.

\subsection{Organization of the Paper}

In Section \ref{section:preliminaries} we discuss the basic definitions and the dynamical formulation of the problem.
Section \ref{section:preparatory results} collects lemmas needed for the proof of the periodicity result.
The results of this section may be read only when needed.
In Section \ref{section:the periodicity result} we prove the main result.
The part of the proof where Bhattacharya's ideas go through by appropriate modification is collected in Section \ref{section:bhattachayas proof in three dimension}.

\subsection{Acknowledgements}
The author would like to thank Siddhartha Bhattacharya, Ankit Rai, Mohit Upmanyu, Nishant Chandgotia, Etienne Mutot, and Pierre Guillon for helpful conversations and their encouragement.
Thanks are due to Prof. Jaikumar Radhakrishnan for helpful suggestions and comments.
This work was carried out while the author was a postdoctoral researcher at the Centre for Applicable Mathematics, Tata Institute of Fundamental Research (TIFR CAM), Bangalore, whose support and hospitality are gratefully acknowledged.

%
%
%
%

%
\section{Preliminaries}
\label{section:preliminaries}

\subsection{Tilings and Periodicity}

\begin{definition}
	\label{definition:tiling}
	A finite subset of $\Z^n$ will be referred to as a \define{cluster}.
	Given a cluster $F\subseteq\Z^n$ we say that $T\subseteq \Z^n$ is an $F$-tiling if for each $p\in \Z^n$ there exist unique $a\in F$ and $t\in T$ such that $p=a+t$.
	In other words, $T$ is an $F$-tiling if and only if $\Z^n$ can be partitioned by $T$-translates of $F$.
	We say that $F$ is \define{exact} if there exists an $F$-tiling.
\end{definition}

\begin{definition}
	\label{definition:periodcity}
	Let $S$ be a subset of $\Z^n$. 
	We say that $S$ is \define{$i$-periodic} if there exists a rank-$i$ subgroup $\Lambda$ of $\Z^n$ such that $S$ is invariant under $\Lambda$, that is, $g+S=S$ for all $g\in \Lambda$.
\end{definition}

\begin{definition}
	\label{definition:weak periodicity}
	A subset $S$ of $\Z^n$ is called \define{$1$-weakly periodic}, or simply \define{weakly periodic} if there exist finitely many $1$-periodic subsets $S_1, \ldots, S_k$ such that $S=S_1\sqcup \cdots \sqcup S_k$.

	A subset $S$ of $\Z^n$ is called \define{$2$-weakly periodic} if there exist finitely many $2$-periodic subsets $S_1, \ldots, S_k$ such that $S=S_1\sqcup \cdots \sqcup S_k$.
\end{definition}

We will use the following results about tilings of $\Z$ and $\Z^2$.

\begin{theorem}
	\label{theorem:uniform periodicity of one dimensional tilings}
	Let $F\subseteq \Z$ be an exact cluster.
	Then there is a positive integer $n$ such that $n+T=T$ for all $F$-tilings $T$.
	In other words, every tiling of $\Z$ is $1$-periodic.
\end{theorem}
\begin{proof}
	Note that if $T|_{[a, b]}$ is known for some interval $[a, b]$ of length exceeding $\text{diam}(F)$, then $T$ is determined entirely.
	Now the assertion follows by a pigeonhole argument.
\end{proof}

\begin{theorem}
	\label{theorem:uniform periodicity of prime cardinality clusters in two dimensions}
	\emph{\cite{szegedy_algorithms_to_tile}, \cite[Example 4]{kari_szabados_alg_geom}}
	Let $F\subseteq\Z^2$ be an exact cluster of prime cardinality such that the affine span of $F$ has rank $2$.
	Then there is a finite index subgroup $\Lambda$ of $\Z^2$ such that every $F$-tiling is $\Lambda$-invariant.
\end{theorem}

\begin{theorem}
	\emph{\cite{bhattacharya_tilings,greenfeld_tao_2020}}
	\label{theorem:existence of biperiodic tilings in two dimensions}
	Let $F\subseteq\Z^2$ be an exact cluster, then there is a $2$-periodic $F$-tiling.
\end{theorem}

\subsection{Dynamical Formulation}

In \cite{bhattacharya_tilings} Bhattacharya proved the periodic tiling conjecture (See \cite{lagarias_wang_tiling}) in two dimensions by developing a dynamical statement which allows one to transfer the problem to an ergodic theoretic setting.
We discuss the relevant definitions and state the dynamical formulation here which we will use later.
The formulation will be given for $\Z^3$.

Let $X$ be a subshift\footnote{There is a natural action of $\Z^3$ on $\set{0, 1}^{\Z^3}$ by translations.
	A \define{subshift} of $\set{0, 1}^{\Z^3}$ is any closed $\Z^3$-invariant subset.} of $\set{0, 1}^{\Z^3}$ and $\mu$ be a $\Z^3$-invariant probability measure on $X$.
The action of $\Z^3$ on $X$ naturally leads to a unitary action of $\Z^3$ on $L^2(X, \mu)$ as follows.
For $\vp\in L^2(X, \mu)$ and $g\in \Z^3$ we define $g\cdot \vp$ by
\begin{equation}
	(g\cdot \vp)(x) = \vp(g^{-1}\cdot x) = \vp((-g)\cdot x)
\end{equation}
The action of $\Z^3$ on $L^2(X, \mu)$ can be extended to an action of $\Z[u_1^{\pm}, u_2^{\pm}, u_3^{\pm}]$--- the ring of Laurent polynomials in three variables with integer coefficients.
To describe this action, first let us set up a convenient notation.
For a vector $v\in \Z^3$, we denote the monomial $u_1^{v_1}u_2^{v_2}u_3^{v_3}$ as $U^v$, and denote the ring $\Z[u_1^\pm, u_2^\pm, u_3^\pm]$ as $\Z[U^\pm]$.
Now any element of $\Z[U^\pm]$ can be written as $\sum_{v\in \Z^3}a_vU^v$, where only finitely many $a_v$'s are nonzero.
We define, for a Laurent polynomial $p=\sum_{v\in \Z^3}a_vU^v$ and $\vp\in L^2(X, \mu)$, the function $p\cdot \vp$ in $L^2(X, \mu)$ as
\begin{equation}
	p\cdot \vp = (\sum_v a_vU^v)\cdot \vp = \sum_{v\in \Z^3} a_v(v\cdot \vp)
\end{equation}
We say that $p$ \define{annihilates} $\vp$ if $p\cdot \vp$ is $0$.

For $i\in \set{1, 2, 3}$, we say that $\vp\in L^2(X, \mu)$ is \define{$i$-periodic} if there is a rank-$i$ subgroup $\Lambda$ of $\Z^3$ such that $v\cdot \vp=\vp$ for all $v\in \Lambda$.\footnote{We emphasize that this equation is written in $L^2$ and hence, when $f$ is an actual function, it only says that $f$ and $v\cdot f$ agree almost everywhere and not necessarily everywhere.}
A measurable subset $B$ of $X$ will be called $i$-periodic if $1_B$ is $i$-periodic. 

Finally, we define a measurable subset $B$ of $X$ to be \define{$2$-weakly periodic} if there exists a partition $B=B_1\sqcup \cdots \sqcup B_k$ of $B$ into finitely many measurable subsets $B_1, \ldots, B_k$ such that each $B_i$ is $2$-periodic.
Similarly, we say that $B$ is $1$-\define{weakly periodic}, or simply \define{weakly periodic} if $B$ can be partitioned into finitely many $1$-periodic subsets.

The following lemma is the $3$-dimensional analog of \cite[Section 2]{bhattacharya_tilings}.

\begin{lemma}
	\emph{\cite[Section 2]{bhattacharya_tilings}}
	\label{label:bhattacharya correspondence principle}
	Write $A=\set{x\in X:\ x(0) = 1}$.
	If $A$ is $i$-periodic then $\mu$-almost every point in $X$ is $i$-periodic.
	If $A$ is $i$-weakly periodic, then $\mu$-almost every point in $X$ is $i$-weakly periodic.
\end{lemma}

\subsection{Spectral Theorem}
\label{section:spectral theorem}

Let $\T^n$ be the $n$-dimensional torus.
Let $\nu$ be a probability measure on $\T^n$.
There is a canonical unitary representation of $\Z^n$ on $L^2(\T^n, \nu)$ which we now describe.
Recall that the characters on $\T^n$ are in bijection with $\Z^n$.
Let us write $\chi_g:\T^n\to \C$ to denote the character corresponding to $g\in \Z^n$.
Then we define a map $\sigma_g:L^2(\T^n, \nu)\to L^2(\T^n, \nu)$ by writing $\sigma_g(\phi) = \chi_g\phi$, where the latter is the pointwise product of $\chi_g$ and $\phi$.
It can be easily checked that each $\sigma_g$ is in fact a unitary linear map.
Thus we get a map $\sigma:\Z^n\to \mc U(L^2(\T^n, \nu))$ which takes $g$ to $\sigma_g$.

By the Stone-Weierstrass theorem we have the $\C$-span of the characters are dense in $C(\T^n)$, where $C(\T^n)$ is the set of all the complex valued continuous functions on $\T^n$ equipped with the sup-norm topology.
Also, since $\T^n$ is a compact metric space, we have $C(\T^n)$ is dense in $L^2(\T^n, \nu)$.
Therefore the $\C$-span of the characters are dense in $L^2(\T^n, \nu)$.
From this we see that, if $\one$ denotes the constant map which takes the value $1$ everywhere, $\text{Span}\set{\sigma_g\one:\ g\in \Z^n}$ is dense in $L^2(\T^n, \nu)$.
In other words, $\one$ is a \emph{cyclic vector} for this representation.
We now want to state a theorem which dictates that this is a defining property of unitary representations of $\Z^n$.

\begin{theorem}
	\label{theorem:spectral theorem}
	\textit{\textbf{Spectral Theorem.}}
	Let $H$ be a Hilbert space and $\tau:\Z^n\to \mc U(H)$ be a unitary representation of $\Z^n$.
	Suppose $v\in H$ is a cyclic vector, that is, $\text{Span}\set{\tau_gv:\ g\in \Z^n}$ is dense in $H$, and assume that $v$ has unit norm.
	Then there is a unique probability measure $\nu$ on $\T^n$ and a unitary isomorphism $\theta: H \to L^2(\T^n, \nu)$ with $\theta(v)= \one$ such that the following diagram commutes for all $g\in \Z^n$
	\begin{figure}[H]
		\centering
		\begin{tikzcd}
			H\ar[r, "\theta"]\ar[d, "\tau_g"] &	L^2(\T^n, \nu)\ar[d, "\sigma_g"]\\
			H\ar[r, "\theta"]\ar[r, "\theta"] & L^2(\T^n, \nu)
		\end{tikzcd}
	\end{figure}
	\noindent
\end{theorem}
So the above theorem says that the abstract representation $\tau$ can be thought of as the canonical concrete representation $\sigma$, at the cost of a probability measure $\nu$.
Thus understanding the measure $\nu$ is equivalent to understanding $\tau$.

\section{Preparatory Lemmas}
\label{section:elementary results}
\label{section:preparatory results}

\subsection{Some Combinatorial Lemmas}

Before stating the lemmas, we fix the meaning of a phrase that recurs throughout this section.
Given a nonzero vector $g\in \R^3$, we say that a plane $\pi\subseteq \R^3$ is \define{parallel} to $g$ if $\pi$ contains a line in the direction of $g$.
Equivalently, the two-dimensional linear subspace obtained by translating $\pi$ to the origin contains $g$.
Equivalently still, $\pi$ is invariant under translation by $g$, that is, $\pi + g = \pi$.
Writing $\pi = \set{x\in \R^3:\ \ab{x, v} = c}$ for some nonzero normal vector $v$ and some scalar $c$, the condition is simply $\ab{g, v} = 0$.
In the same way, a line $\ell\subseteq \R^3$ is parallel to $g$ if it lies in the direction of $g$, i.e.\ $\ell + g = \ell$.

\begin{lemma}
	\label{lemma:first combonatorial lemma precursor}
	Let $F$ be a finite subset of $\Z^3$ and $p$ be a prime.
	Let $g$ be a nonzero vector in $\Z^3$.
	Assume that whenever $\pi$ is a plane in $\R^3$ parallel to $g$ we have $|\pi\cap F|$ is divisible by $p$.
	Then whenever $\ell$ is a line parallel to $g$ we have $|\ell\cap F|$ is also divisible by $p$. (The converse is also true.)
\end{lemma}
\begin{proof}
	We may assume that $g$ is primitive.\footnote{See Appendix \ref{section:algebra appendix} for the definition of a primitive vector.} 
	Using Lemma \ref{lemma:glnz acts transitively on the set of all the primitive vectors}, after applying a suitable $GL_3(\Z)$ transformation, we may assume that $g=(0, 0, 1)$.
	Let $S$ be the image of $F$ under the map $f:\R^3\to \R^2$ defined by
	\begin{equation}
		f(x, y, z) = (x, y)
	\end{equation}
	for all $(x, y, z)\in \Z^3$, that is, $f$ is the orthogonal projection on the $x, y$-plane.
	Let $(a, b)$ be an extreme point of the convex hull of $S$ in $\R^2$ and let $\ell$ be a line in $\R^2$ such that $\ell\cap S = \set{(a, b)}$.
	Let $\pi$ be a plane parallel to $g$ in $\R^3$ such that the image of $\pi$ under $f$ is $\ell$.
	Then $\pi\cap F = \set{(a, b, z)\in \Z^3:\ (a, b, z)\in F}$.
	By hypothesis, $|\pi\cap F|$ has size divisible by $p$.
	Now the set $F\setminus \pi\cap F$ satisfies the same hypothesis as $F$ and now we can finish inductively.

	For the converse, suppose that $|\ell\cap F|$ is divisible by $p$ for every line $\ell$ parallel to $g$.
	Any plane $\pi$ parallel to $g$ is partitioned by the lines parallel to $g$ contained in it, so $\pi\cap F$ is the disjoint union of the finitely many nonempty sets $\ell\cap F$ with $\ell\subseteq \pi$.
	Each such set has size divisible by $p$, and hence so does $|\pi\cap F|$.
\end{proof}

\begin{definition}
	Let $S$ be a subset of $\Z^3$.
	Let $\Lambda$ be a rank-$2$ subgroup of $\Z^3$ and $g$ be a nonzero vector such that $\Z g + \Lambda$ is a rank-$3$ subgroup of $\Z^3$ .
	We say that $S$ is a \define{prism} with \define{base} $\Lambda$ and \define{axis} $g$ if there is a vector $h\in \Z^3$ and a finite set $\set{0=n_0, n_1, \ldots, n_k}$ of integers such that 
	\begin{equation}
		S = \bigcup_{i=0}^k ( n_ig\ +\ (h+\Lambda) \cap S )
	\end{equation}
	We refer to $(h+\Lambda)\cap S$ as a \define{foundation} of $S$.
\end{definition}

An example of a prism would be any set of the form $A\times B$ where $A\subseteq\Z^2$ and $B\subseteq \Z$.

\begin{lemma}
	\label{lemma:first combinatorial lemma}
	Let $F\subseteq \Z^3$ be a set of size $p^2$, where $p$ is a prime.
	Suppose that there is a nonzero vector $g$ in $\Z^3$ such that whenever $\pi$ is a plane parallel to $g$ in $\R^3$ we have $|\pi\cap F|$ is divisible by $p$.
	Also, assume that there is a plane $\pi'$, not parallel to $g$ and parallel to a rank-$2$ subgroup of $\Z^3$, such that $|\pi''\cap F|$ is divisible by $p$ whenever $\pi''$ is a plane parallel to $\pi'$.
	Then $F$ is a prism with foundation of size $p$.
\end{lemma}
\begin{proof} 
	Since $\pi'$ is parallel to a rank-$2$ subgroup of $\Z^3$, the two-dimensional subspace parallel to $\pi'$ is the real span of a rank-$2$ (primitive) sublattice of $\Z^3$.
	Using Lemma \ref{lemma:flattening subgroup lemma}, after a suitable $GL_3(\Z)$ transformation, we may assume that $\pi'$ is the plane $\Lambda = \Z^2\times \set{0}$, so that the planes parallel to $\pi'$ are exactly the $\Z^2\times\set{n}$, $n\in\Z$.
	Since $\pi'$ is not parallel to $g$ by hypothesis, we see that $\Z g + \Lambda$ is a subgroup of rank $3$ in $\Z^3$.
	Also, without loss of generality assume that $g$ is a primitive vector.
	By translating $F$ if necessary, we may assume that $\Lambda$ intersects $F$ but $\Z^2\times \set{k}$ does not intersect $F$ whenever $k$ is a negative integer.

	Define $B = \set{n\geq 0:\ (\Z^2\times\set{n})\cap F\neq \emptyset }$ and say $B = \set{0=b_0, b_1, \ldots, b_k}$.
	Since, by hypothesis, each translate of $\Lambda$ intersects $F$ in a set of size divisible by $p$, we see that $(k+1)p = |B|p\leq |F| = p^2$ and hence $k+1\leq p$.
	By Lemma \ref{lemma:first combonatorial lemma precursor} we know that whenever $\ell$ is a line in $\R^3$ parallel to $g$, we have that $\ell\cap F$ has size divisible by $p$.
	Thus there are at least $p$ translates of $\Lambda$ in the direction of $g$ which intersect $F$.
	Thus $k+1 = |B|\geq p$.
	So we have $k+1 = p$.
	This forces that $(\Z^2\times \set{b_i})\cap F$ has size $p$ for each $i\in \set{0, 1, \ldots, k}$.

	Since $\pi'=\Lambda$ is not parallel to $g$, every line parallel to $g$ meets each level $\Z^2\times\set{b_i}$ at most once, so a line parallel to $g$ that meets $F$ contains at most $k+1=p$ points.
	Being a nonzero multiple of $p$ by Lemma \ref{lemma:first combonatorial lemma precursor}, it contains exactly $p$ points, one in $F\cap(\Z^2\times\set{b_i})$ for each $i$.
	In particular, the line parallel to $g$ through any point of $\Lambda\cap F$ meets $F$ at each level $b_i$.
	Let $p_0, p_1, \ldots, p_k\in \Z^2$ be such that $(p_i, b_i) - (p_0, b_0)$ are parallel to $g$ for each $i\in \set{1, \ldots, k}$.
	Since $g$ is a primitive vector, we can find integers $n_1, \ldots, n_k$ such that $n_ig = (p_i, b_i) - (p_0, b_0)$.
	It follows that
	\begin{equation}
		F = \bigcup_{i=0}^k (n_ig + (\Lambda\cap F))
	\end{equation}
	and hence $F$ is a prism with base $\Lambda$ and axis $g$, and foundation of size $p$.
\end{proof}

A sublattice $L$ of $\Z^n$ is called \define{saturated} if it contains every $z\in\Z^n$ some positive integer multiple of which lies in $L$.
Equivalently, $L=\Z^n\cap\text{span}_\R(L)$, or again, the quotient $\Z^n/L$ is torsion-free.
A saturated sublattice is a direct summand of $\Z^n$, so it admits a complementary subgroup $V$ with $\Z^n=L+V$ and $L\cap V=\set{0}$.

\begin{lemma}
	\label{lemma:tiling by a prism with prime base}
	Let $F\subseteq\Z^3$ be a prism with foundation of prime size.
	Assume that $F$ is an exact cluster.
	Then there is a $3$-periodic $F$-tiling.
\end{lemma}
\begin{proof}
	Suppose first that the affine span of $F$ has dimension at most $2$.
	After a translation we may assume $0\in F$, so that $F$ is contained in the sublattice $L=\Z^3\cap\operatorname{span}_\R(F)$, of rank $d\le 2$.
	This $L$ is saturated.
	We identify it with $\Z^d$.
	Since $F\subseteq L$, a tile $F+t$ can meet $L$ only when $t\in L$, in which case $F+t\subseteq L$.
	Hence the tiles of any $F$-tiling that meet $L$ already tile $L$, so $F$ is an exact cluster of $L$.
	By Theorem \ref{theorem:existence of biperiodic tilings in two dimensions} (if $d=2$), Theorem \ref{theorem:uniform periodicity of one dimensional tilings} (if $d=1$), or trivially (if $d=0$), there is an $F$-tiling $W_0$ of $L$ invariant under a finite-index subgroup of $L$.
	Choose a complementary subgroup $V$ of $\Z^3$ so that $\Z^3=L\oplus V$ (possible since $L$ is saturated), so that every element of $\Z^3$ is uniquely a sum $\ell+v$ with $\ell\in L$ and $v\in V$.
	Since $W_0$ is an $F$-tiling of $L$, each such $\ell$ is in turn uniquely $a+w$ with $a\in F$ and $w\in W_0$.
	Hence every element of $\Z^3$ is uniquely $a+(w+v)$ with $a\in F$ and $w+v\in W_0+V$, so $W_0+V$ is an $F$-tiling of $\Z^3$.
	It is invariant under $V$ together with the finite-index subgroup of $L$ fixing $W_0$, hence under a finite-index subgroup of $\Z^3$, so it is $3$-periodic.
	From now on we assume that the affine span of $F$ has dimension $3$.

	Let $\Lambda$ and $g$ be the base and axis of $F$ respectively, and let $A=(h+\Lambda)\cap F$ be a foundation, so that $|A|$ is the prime in question and $F=\bigcup_{i=0}^k\bigl(n_ig+A\bigr)$ with $n_0=0$.
	After translating $F$ we may take $h=0$, so that $A=\Lambda\cap F$.

	Let $M=\Lambda + \Z g$, a rank-$3$ (hence finite-index) subgroup of $\Z^3$.
	The sum is direct since $g\notin\text{span}_{\R}(\Lambda)$.
	Fixing a basis $\set{\lambda_1,\lambda_2}$ of $\Lambda$, the triple $\set{\lambda_1,\lambda_2,g}$ is a basis of $M$.
	In the resulting coordinates $M\cong\Z^3$ we have $\Lambda=\Z^2\times\set{0}$, $g=(0,0,1)$, and, since $A\subseteq\Lambda$,
	\[
		F=\bigcup_{i=0}^k\bigl(n_ig+A\bigr)=A\times N,\qquad N=\set{n_0,\ldots,n_k}\subseteq\Z .
	\]
	Thus, inside $M$, the prism $F$ is the product $A\times N$.
	Since $F\subseteq M$ and the finitely many cosets of $M$ partition $\Z^3$, a point of a coset $M+r$ can be covered only by tiles $F+t$ with $t\in M+r$. Hence any $F$-tiling of $\Z^3$ restricts to an $F$-tiling of each coset, and conversely, if $T^\ast$ is an $F$-tiling of $M$ and $R$ is a set of coset representatives, then $\bigcup_{r\in R}(T^\ast+r)$ is an $F$-tiling of $\Z^3$ invariant under every subgroup of $M$ that leaves $T^\ast$ invariant.
	It therefore suffices to find a $3$-periodic $F$-tiling of $M$: a rank-$3$ subgroup of $M$ has finite index in $\Z^3$, so the resulting tiling of $\Z^3$ is again $3$-periodic.
	Renaming $M$ as $\Z^3$, we may assume $F=A\times N$ with $A\subseteq\Z^2$ and $N\subseteq\Z$ finite.

	\emph{$A$ and $N$ are exact.}
	Fix any $F$-tiling $T$ of $\Z^3$ and write $t=(t_{xy},t_z)$ for $t\in T$.
	A tile $(A\times N)+t$ meets the plane $\Z^2\times\set{c}$ in $(A+t_{xy})\times\set{c}$ when $c-t_z\in N$, and is disjoint from it otherwise.
	As these intersections partition $\Z^2\times\set{c}$, the set $\set{t_{xy}:\ c-t_z\in N}$ is an $A$-tiling of $\Z^2$.
	In particular $A$ is exact.
	Likewise, fixing $(x,y)\in\Z^2$, the tiles meeting the column $\set{(x,y)}\times\Z$ are exactly those with $(x,y)\in A+t_{xy}$, and such a tile covers the levels $N+t_z$.
	As these partition $\Z$, the set $\set{t_z:\ (x,y)\in A+t_{xy}}$ is an $N$-tiling of $\Z$, so $N$ is exact.

	\emph{Construction.}
	Since $F=A\times N$ has $3$-dimensional affine span and the affine dimension of a product is the sum of those of its factors, $A$ has $2$-dimensional affine span (and $N$ has $1$-dimensional affine span).
	Thus $A$ is an exact cluster of prime cardinality with rank-$2$ affine span, and Theorem \ref{theorem:uniform periodicity of prime cardinality clusters in two dimensions} provides an $A$-tiling $S$ of $\Z^2$ invariant under a rank-$2$ subgroup $\Gamma$ of $\Z^2$.
	By Theorem \ref{theorem:uniform periodicity of one dimensional tilings}, the exact cluster $N\subseteq\Z$ admits a ($1$-periodic) $N$-tiling $W$ of $\Z$.
	Since $S$ is an $A$-tiling of $\Z^2$ and $W$ is an $N$-tiling of $\Z$, the product $S\times W$ is an $F$-tiling of $\Z^3$: each $(x,z)\in\Z^2\times\Z$ has $x=a+s$ and $z=n+w$ for unique $a\in A$, $s\in S$, $n\in N$, $w\in W$, that is, $(x,z)=(a,n)+(s,w)$ for a unique $(a,n)\in F=A\times N$ and a unique $(s,w)\in S\times W$.
	Now $S$ is invariant under the rank-$2$ subgroup $\Gamma$ of $\Z^2$ and, being $1$-periodic, $W$ is invariant under $n_0\Z$ for some positive integer $n_0$.
	Hence $S\times W$ is invariant under $\Gamma\times n_0\Z$, a rank-$3$ (finite-index) subgroup of $\Z^3$, so it is a $3$-periodic $F$-tiling, completing the proof.
\end{proof}

\begin{lemma}
	\label{lemma:dilation lemma for higher level tilings}
	\textit{\textbf{Dilation Lemma.}}
	\emph{\cite[Corollary 11]{horak_kim_algebraic_method}}
	Let $F\subseteq \Z^3$.
	If $T$ is an $F$-tiling then for all $\alpha$ relatively prime to $|F|$ we have $T$ is also an $\alpha F$-tiling, where $\alpha F=\set{\alpha a:\ a\in F}$.
\end{lemma}

It should be noted that various authors (\cite{tijdeman_decomposition_of}, \cite{szegedy_algorithms_to_tile}, \cite{kari_szabados_alg_geom}, \cite{bhattacharya_tilings}, \cite{greenfeld_tao_2020}) had discovered the above lemma in one form or another.

%
%
%

\subsection{An Analytical Lemma}

\begin{definition}
	An element $\gamma\in S^1$ is said to be \define{irrational} if it is equal to $e^{i\theta}$ for some $\theta$ which is an irrational multiple of $2\pi$ .
	Equivalently, $\gamma\in S^1$ is irrational if there is no non-trivial character of $S^1$ in whose kernel $\gamma$ lies.
	More generally, elements $\gamma_1, \ldots, \gamma_m$ in $S^1$ are called \define{rationally independent} if there is no non-trivial character $\chi$ of $(S^1)^m$ such that $\chi(\gamma_1, \ldots, \gamma_m) = 1$.
\end{definition}

\begin{lemma}
	\label{lemma:non measure preparatory lemma}
	Let $\gamma_1, \ldots, \gamma_n$ be irrational elements in $S^1$ and $x_1, \ldots, x_n$ be complex numbers.
	Assume that
	\begin{equation}
		(\gamma_1^k - 1) x_1 + \cdots + (\gamma_n^k - 1) x_n \in \Z
	\end{equation}
	for all non-negative integers $k$.
	Then the above expression is $0$ for all $k$.
\end{lemma}
\begin{proof}
	Write
	\[
		S_k=\sum_{i=1}^n(\gamma_i^k-1)x_i.
	\]
	we must show that $S_k=0$ for every $k\geq 0$.
	Let $G=\ab{\gamma_1,\ldots,\gamma_n}$ be the subgroup of $S^1$ generated by the $\gamma_i$.
	It is finitely generated and abelian, and its torsion subgroup $T$ is finite.
	Since every finite subgroup of $S^1$ is cyclic, we may write $T=\ab{\zeta}$ for a primitive $q$-th root of unity $\zeta$ (with $q=1$ when $T$ is trivial).
	The quotient $G/T$ is finitely generated and torsion-free, hence free.
	Choosing $\beta_1,\ldots,\beta_d\in G$ whose images form a basis of $G/T$, every element of $G$, and in particular each $\gamma_i$, can be written as
	\[
		\gamma_i=\zeta^{c_{i0}}\beta_1^{c_{i1}}\cdots\beta_d^{c_{id}},
		\qquad c_{ij}\in\Z .
	\]
	The elements $\beta_1,\ldots,\beta_d$ are rationally independent.
	This is because if $\beta_1^{a_1}\cdots\beta_d^{a_d}=1$ for some integers $a_1, \ldots, a_d$, then that fact that $\beta_1, \ldots, \beta_d$ form a basis of $G/T$ yields $a_1=\cdots=a_d=0$.
	Set $c_i=(c_{i1},\ldots,c_{id})\in\Z^d$.
	Since each $\gamma_i$ is irrational, it is not a root of unity, so $c_i\neq 0$ for every $i$ (otherwise $\gamma_i=\zeta^{c_{i0}}$ would lie in $T$).

	Fix a residue $\rho\in\set{0,1,\ldots,q-1}$ and define a Laurent polynomial on $(S^1)^d$ by
	\[
		f_\rho(U_1,\ldots,U_d)
		=
		\sum_{i=1}^n\lrp{\zeta^{\rho c_{i0}}\beta_1^{\rho c_{i1}}\cdots\beta_d^{\rho c_{id}}\,U_1^{c_{i1}}\cdots U_d^{c_{id}}-1}x_i .
	\]
	Put $\delta_j=\beta_j^{\,q}$ for $1\leq j\leq d$.
	The $\delta_j$ are again rationally independent, since $\delta_1^{a_1}\cdots\delta_d^{a_d}=1$ gives $\beta_1^{qa_1}\cdots\beta_d^{qa_d}=1$, whence $qa_j=0$ and so $a_j=0$ for all $j$.
	For $k=\rho+qt$ with $t\geq 0$ we have $\zeta^{kc_{i0}}=\zeta^{\rho c_{i0}}$ and $\beta_j^{kc_{ij}}=\beta_j^{\rho c_{ij}}\delta_j^{tc_{ij}}$, so that
	\[
		f_\rho(\delta_1^t,\ldots,\delta_d^t)=S_{\rho+qt}\in\Z
		\qquad\text{for all }t\geq 0 ,
	\]
	where the integrality comes from the hypothesis.
	By the rational independence of $\delta_1,\ldots,\delta_d$, the set $\set{(\delta_1^t,\ldots,\delta_d^t):\ t\geq 0}$ is dense in $(S^1)^d$.\footnote{See Theorem 4.14 in \cite{einsiedler_ward_ergodic_theory}.}
	As $f_\rho$ is continuous and integer-valued on this dense set, its image lies in the closed set $\Z$.
	Further, as $(S^1)^d$ is connected while $\Z$ is discrete, $f_\rho$ is constant.
	Denote its value by $a_\rho$.
	Evaluating at $t=0$ gives $a_\rho=S_\rho$.

	It remains to show that $a_\rho=0$ for every $\rho$.
	Recall that distinct characters of $(S^1)^d$ are linearly independent, so a Laurent polynomial is determined by its monomial coefficients.
	In particular a constant Laurent polynomial equals its constant coefficient, and one vanishing on $(S^1)^d$ has all its coefficients zero.

	Consider first $\rho=0$. Here
	\[
		f_0(U)=\sum_{i=1}^n\lrp{U_1^{c_{i1}}\cdots U_d^{c_{id}}-1}x_i ,
	\]
	which is constant with value $f_0(1,\ldots,1)=0$, so $f_0$ vanishes identically and its constant coefficient is $0$.
	Because every $c_i\neq 0$, each monomial $U_1^{c_{i1}}\cdots U_d^{c_{id}}$ is nonconstant, so the only constant contribution to $f_0$ comes from the terms $-x_i$.
	Hence its constant coefficient equals $-\sum_{i=1}^n x_i$, and we conclude that
	\[
		\sum_{i=1}^n x_i=0 .
	\]

	Now take an arbitrary $\rho$. Since every $c_i\neq 0$, every monomial appearing in $f_\rho$ is likewise nonconstant, so the only constant contribution to $f_\rho$ again comes from the terms $-x_i$. Thus the constant coefficient of $f_\rho$ is $-\sum_{i=1}^n x_i=0$. As $f_\rho$ is constant, it equals its constant coefficient, and therefore $a_\rho=0$.

	Finally, for any $k\geq 0$, writing $k=\rho+qt$ yields $S_k=f_\rho(\delta_1^t,\ldots,\delta_d^t)=a_\rho=0$, as claimed.
\end{proof}


%
%
%

%
\subsection{An Elementary Fact about the Kernel of Characters}

\begin{lemma}
	\label{lemma:intersection of two kernels}
	Let $g$ and $h$ be nonzero vectors in $\Z^3$ which are linearly independent.
	Then there is $v\in \Z^3\setminus \set{0}$ orthogonal to both $g$ and $h$ and a positive integer $n$ such that 
	\begin{equation}
		\ker\chi_g\cap \ker\chi_h
		\subseteq
		\bigcup_{0\leq i, j, k< n} \lrb{ \lrp{ \frac{i}{n}, \frac{j}{n}, \frac{k}{n}} + \set{tv\in \R^3/\Z^3:\ t\in \R}}
	\end{equation}
\end{lemma}
\begin{proof}
	Let $g = (g_1, g_2, g_3)$ and $h=(h_1, h_2, h_3)$.
	Consider the $2\times 3$-matrix $M$ over $\Q$ defined as
	\begin{equation}
		M
		=
		\begin{bmatrix}
			g_1 & g_2 & g_3\\
			h_1 & h_2 & h_3
		\end{bmatrix}
	\end{equation}

	Note that an element $z$ of $\T^3 = \R^3/\Z^3$ is in $\ker\chi_g\cap \ker\chi_h$ if and only if some (any) representative $z'$ of $z$ in $\R^3$ satisfies $Mz'\in \Z^2$.
	Since $g$ and $h$ are linearly independent, the rank of $M$ is $2$, and its nullity is therefore $1$.
	Thus, since $M$ is a matrix over $\Q$, we can find a nonzero vector $v$ in $\Q^3$ such that $Mv = 0$.
	Clearly, any such vector $v$ also spans the null-space of $M$ since the null-space of $M$ is $1$-dimensional.
	By scaling $v$ if necessary, we may assume that $v\in \Z^3$.

	Let $\Lambda = M(\Z^3)$, that is, $\Lambda\subseteq \Z^2$ is the image of $\Z^3$ under $M$.
	Since $M$ has rank $2$, we see that $\Lambda$ is a finite index subgroup of $\Z^2$.
	Let $n$ be the smallest positive integer such that $(n, 0)$ and $(0, n)$ are both in $\Lambda$.
	Let $u_1$ and $u_2$ in $\Z^3$ be such that $Mu_1 = (n, 0)$ and $Mu_2 = (0, n)$.
	We will show that
	\begin{equation}
		\ker\chi_g\cap \ker\chi_h
		\subseteq
		\bigcup_{0\leq i, j, k< n} \lrb{ \lrp{ \frac{i}{n}, \frac{j}{n}, \frac{k}{n}} + \set{tv\in \R^3/\Z^3:\ t\in \R}}
	\end{equation}

	Let $w\in \ker\chi_g\cap \ker\chi_h$ be arbitrary.
	Let $w'$ be a representative of $w$ in $\R^3$.
	Then $Mw' = (a, b)$ for some $(a, b)\in \Z^2$.
	Therefore
	\begin{equation}
		Mw' = \frac{a}{n} Mu_1 + \frac{b}{n} Mu_2
	\end{equation}
	giving
	\begin{equation}
		M\lrp{w' - \frac{a}{n}u_1 - \frac{b}{n} u_2} = 0
	\end{equation}
	Thus $w' - (a/n) u_1 - (b/n) u_2$ is in the null-space of $M$, and hence there is $t\in \R$ such that $w' - (a/n) u_1 - (b/n) u_2 = tv$.
	Therefore $w' = (au_1+ bu_2)/n + tv$.
	We can find $0\leq i, j, k< n$ such that $(au_1 + bu_2)/n \equiv (i/n, j/n, k/n) \pmod{\Z^3}$.
	Therefore
	$$
	w' \equiv \lrp{ \frac{i}{n}, \frac{j}{n}, \frac{k}{n}} + tv\pmod{\Z^3}
	$$
	showing the desired containment.
\end{proof}

\subsection{Algebraic Lemmas}

The use of cyclotomic-polynomial divisibility to constrain a tile is classical.
In dimension one it goes back to complementing pairs $A\oplus B=\Z$, where Tijdeman \cite{tijdeman_decomposition_of} used cyclotomic generating functions (the vanishing of $\sum_{a\in A}\zeta^{a}$ at roots of unity and divisibility by the cyclotomic factors of $z^N-1$) to prove that when $|A|$ is a prime power every element of the complement is divisible by that prime.
The relevant divisibility conditions were later distilled by Coven and Meyerowitz \cite{coven_meyerowitz_tiling_integers}, whose criteria completely characterize tilings of $\Z$ by a single finite set when the cardinality has at most two prime factors.
We too rely on cyclotomic divisibility but to a different end.
Rather than classifying tilings of $\Z$ directly, we apply it to a cluster $F\subseteq\Z^3$, forcing divisibility by $p$ of the cardinalities of certain slices of $F$, a structural constraint that feeds the ergodic-theoretic argument of the later sections.

\begin{lemma}
	\label{lemma:cyclotomy divisibility lemma}
	Let $p$ be a prime and $k$ be a positive integer.
	Let $\zeta$ be a primitive $p^k$-th root of unity.
	Suppose there are integers $a_1, \ldots, a_m$ such that
	\begin{equation}
		\zeta^{a_1} + \cdots + \zeta^{a_m} = 0
	\end{equation}
	Then $p$ divides $m$.
\end{lemma}
\begin{proof}
	Without loss of generality we may assume that the $a_i$'s are all non-negative.
	Define the polynomial $Q(z) = z^{a_1} + \cdots + z^{a_m}$.
	Then since $\zeta$ is a root of $Q(z)$, we have $Q(z)$ is divisible by the $p^k$-th cyclotomic polynomial $\Phi_{p^k}(z)$.
	Let $\Psi(z)$ be an integer polynomial such that $Q(z) = \Phi_{p^k}(z)\Psi(z)$.
	Now using the fact that $\Phi_{p^k}(z) = \Phi_p(z^{p^{k-1}})$ we have $Q(z) = \Phi_p(z^{p^{k-1}})\Psi(z)$.
	Now substituting $z=1$ gives $m = p \Psi(1)$ and hence $p$ divides $m$.
\end{proof}

\begin{definition}
	For any finite subset $S$ of $\Z^3$, we will write $Z(\sum_{g\in S} \chi_{g})$ to mean the set of all the points in $\T^3$ on which $\sum_{g\in S}\chi_{g}$ vanishes.
\end{definition}

\begin{lemma}
	\label{lemma:first spectral support lemma}
	\emph{(See \cite[Lemma 3.2]{bhattacharya_tilings})}
	Let $F\subseteq \Z^3$ be a finite set containing the origin.
	Then there is a finite subset $\Delta$ of $\Z^3\setminus \set{0}$ with pairwise linearly independent elements such that
	\begin{equation}
		Z:=
		\bigcap_{\alpha\text{ coprime to } |F|}
		Z\lrb{\sum_{g\in F} \chi_{\alpha g}}
		\quad \subseteq\quad
		\bigcup_{h\in \Delta} \ker\chi_h
	\end{equation}
\end{lemma}
\begin{proof}
	Let $n$ be the product of all primes which divide $|F|$.
	Then for each non-negative integer $k$ we have $nk + 1$ is relatively prime to $|F|$.
	Fix $(a, b, c)\in Z$.
	Then we have, for all non-negative integers $k$ that
	\begin{equation}
		\sum_{g\in F} \chi_{(nk+1)g}(a, b, c) = 0
	\end{equation}
	giving
	\begin{equation}
		\sum_{g\in F\setminus\set{0}} \chi_{(nk+1)g}(a, b, c)
		=
		\sum_{g\in F\setminus\set{0}} a^{g_1}b^{g_2}c^{g_3}(a^{ng_1}b^{ng_2}c^{ng_3})^k
		=
		-1
	\end{equation}
	Therefore
	\begin{equation}
		\sum_{g\in F\setminus\set{0}} a^{g_1}b^{g_2}c^{g_3}\lrb{ \frac{1}{N}\sum_{k=0}^{N-1}
			(a^{ng_1}b^{ng_2}c^{ng_3})^k
		}
		=
		-1
	\end{equation}
	for all $N\geq 1$.
	Taking the limit $N\to \infty$ we see that there must exist $g\in F\setminus{\set{0}}$ such that $a^{ng_1}b^{ng_2}c^{ng_3}$ is $1$,\footnote{
		This is because if $z\in S^1$ then the average $(z+z^2+ \cdots + z^N)/N$ does not converge to zero if and only if $z=1$.
	}
	and thus $(a, b, c)\in \ker(\chi_{ng})$.
	Therefore $Z\subseteq \bigcup_{g\in F\setminus \set{0}} \ker\chi_{ng}$.

	Let $\Delta$ be a non-empty subset of $\Z^3\setminus \set{0}$ of smallest possible size such that $Z\subseteq \bigcup_{g\in \Delta} \ker\chi_{g}$.
	Such a $\Delta$ exists by the above paragraph.
	We claim that the elements of $\Delta$ are pairwise linearly independent.
	Suppose not.
	Then there exist distinct $g, h\in \Delta$ such that $g$ and $h$ are linearly dependent.
	We can thus find a nonzero vector $v$ such that $v \in (\Z g)\cap (\Z h)$.
	Let $\Delta' = (\Delta\setminus \set{g, h}) \cup\set{v}$.
	It is clear that 
	\begin{equation}
		\bigcup_{u\in \Delta} \ker\chi_u
		\subseteq
		\bigcup_{u\in \Delta'}\ker\chi_u
	\end{equation}
	and hence $Z \subseteq \bigcup_{u\in \Delta'}\ker\chi_u$.
	But $\Delta'$ has size strictly smaller than the size of $\Delta$ which is a contradiction to the choice of $\Delta$.
	This finishes the proof.
\end{proof}

\begin{lemma}
	\label{lemma:second spectral support lemma}
	Let $F$ be a cluster in $\Z^3$ containing the origin and $p$ be a prime.
	Assume $|F|$ is a power of $p$.
	Let $h$ be an arbitrary nonzero vector in $\Z^3$. 
	Then at least one of the following must happen.
	\begin{enumerate}
		\item Every line in $\R^3$ that is parallel to $h$ intersects $F$ in a set of size divisible by $p$.
		\item There is a finite set $\Gamma\subseteq \Z^3\setminus \set{0}$ such that each element of $\Gamma$ is linearly independent with $h$ and
			\begin{equation}
				\ker\chi_h
				\cap
				\lrp{
					\bigcap_{\alpha\text{ coprime to } |F|}
					Z\lrb{\sum_{g\in F} \chi_{\alpha g}}
				}
				\subseteq
				\bigcup_{v\in \Gamma} \ker\chi_h\cap \ker\chi_v
			\end{equation}
	\end{enumerate}
\end{lemma}
\begin{proof}
	Using Lemma \ref{lemma:glnz acts transitively on the set of all the primitive vectors}, after applying a suitable $GL_3(\Z)$ transformation, we may assume that $h=(0, 0, m)$ for some positive integer $m$.\footnote{
		This is because of the following.
		For any linear map $T:\R^3\to \R^3$ such that $T\in GL_3(\Z)$, and any $F\subseteq\Z^3$ finite, we have $Tp\in Z\lrp{\sum_{g\in F} \chi_g}$ if and only if $(T^{-1})^*p\in Z\lrp{\sum_{g\in T(F)}\chi_{g}}$.
	}
	Let $(a, b, c)$ be an arbitrary element in $\ker\chi_h\cap Z$, where
	\begin{equation}
		Z=\bigcap_{\alpha\text{ coprime to } |F|} Z\lrb{\sum_{g\in F} \chi_{\alpha g}}
	\end{equation}
	Let $h_0=(0, 0, 1)$.
	Then
	\begin{equation}
		\ker\chi_h = \bigcup_{\omega\in S^1:\ \omega^m=1} (1, 1, \omega)\cdot \ker_{\chi_{h_0}}
	\end{equation}
	Thus there is an $m$-th root of unity $\omega$ such that $(a, b, c) \in ((1, 1, \omega)\cdot \ker\chi_{h_0})\cap Z$.
	Then $c=\omega$ and hence
	\begin{equation}
		\sum_{g\in F} \chi_{\alpha g}(a, b, \omega) = 0
	\end{equation}
	for all $\alpha$ relatively prime to $|F|$.

	Let $d$ be a positive integer such that $\omega$ is a primitive $d$-th root of unity.
	Let $d = p^r\beta$ where $\beta$ is relatively prime to $p$. 
	Let $n$ be the product of all the primes dividing $|F|$.
	For all non-negative integers $k$, we have $mn k + 1$ is relatively prime to $|F|$.
	Using the fact that $|F|$ is a power of $p$, we have $(mnk+1)\beta$ is also relatively prime to $|F|$.
	So we have, by substituting $\alpha =(mnk+1)\beta$, that
	\begin{equation}
		\label{equation:conundrum}
		\sum_{g\in F} \chi_{(mnk+1)\beta g}(a, b, \omega)
		=
		\sum_{g\in F} a^{\beta g_1} b^{\beta g_2}\omega^{\beta g_3} \chi_{mnk\beta(g_1, g_2)}(a, b)
		=
		0
	\end{equation}
	for all non-negative integers $k$.
	Let $\zeta = \omega^\beta$, and hence $\zeta$ is a primitive $p^r$-th root of unity.
	Let $\sim$ be a relation on $F$ defined as $g\sim g'$ for $g, g'\in F$ if and only if $(g_1, g_2) = (g_1', g_2')$.
	Then $\sim$ is an equivalence relation, and let $\mc E$ be the set of all the equivalence classes of $\sim$.\footnote{Note that the equivalence classes are precisely the sets that are obtained by intersecting $F$ with a line parallel to $h$.}
	For each $E\in \mc E$ choose an element $g^E$ of $E$.
	For any $v\in \Z^3$ let $\hat v$ denote the vector in $\Z^2$ obtained by dropping the last coordinate of $v$.
	Then Equation \ref{equation:conundrum} gives
	\begin{equation}
		\label{equation:second conundrum}
		\sum_{E\in \mc E}\lrb{ \lrp{\sum_{g\in E} \zeta^{g_3}}\chi_{\beta \hat g^E}(a, b) \chi_{mnk\beta \hat g^E}(a, b)}
		=
		0
	\end{equation}
	for all $k$ non-negative.
	Let $\theta_E = \sum_{g\in E} \zeta^{g_3}$ and $\gamma_E= \theta_E \chi_{\beta \hat g^E}(a, b)$ for all $E$ in $\mc E$.
	If $\theta_E$ is $0$ for all $E\in \mc E$, then by Lemma \ref{lemma:cyclotomy divisibility lemma} we have each $|E|$ is divisible by $p$.
	This would mean precisely that (1) holds. 
	So we may assume that there is some $E_0\in \mc E$ such that $\theta_{E_0}$, or equivalently $\gamma_{E_0}$, is nonzero.
	Then from Equation \ref{equation:second conundrum} we have
	\begin{equation}
		\chi_{-mnk \beta\hat g^{E_0}}(a, b)\lrb{\sum_{E\in \mc E} \gamma_E \chi_{mnk\beta \hat g^E}(a, b)}
		=
		0
	\end{equation}
	which gives
	\begin{equation}
		\gamma_{E_0} + \sum_{E\in \mc E\setminus\set{E_0}} \gamma_E \chi_{mnk\beta (\hat g^E - \hat g^{E_0})}(a, b)
		=
		\gamma_{E_0} + \sum_{E\in \mc E\setminus\set{E_0}} \gamma_E \chi_{mn\beta (\hat g^E - \hat g^{E_0})}(a, b)^k
		=
		0
	\end{equation}
	Averaging we have
	\begin{equation}
		\gamma_{E_0} + \sum_{E\in \mc E\setminus\set{E_0}} \gamma_E\lrb{ \frac{1}{N}\sum_{k=0}^{N-1}\chi_{mn\beta (\hat g^E - \hat g^{E_0})}(a, b)^k}
		=
		0
	\end{equation}
	Taking limit $N\to \infty$, we must have, since $\gamma_{E_0}\neq 0$, that $(a, b)\in \ker\chi_{mn\beta (\hat g^{E} - \hat g^{E_0})}$ for some $E\neq E_0$ in $\mc E$.
	Set
	\begin{equation}
		h^E = mn\beta\,(g^{E}_1-g^{E_0}_1,\ g^{E}_2 - g^{E_0}_2,\ 1).
	\end{equation}
	Since $\hat g^E - \hat g^{E_0}\neq 0$, this vector is linearly independent with $h=(0,0,m)$, and $(a,b,c)\in\ker\chi_{h^E}$: indeed
	\begin{equation}
		\chi_{h^E}(a,b,c) = \chi_{mn\beta(\hat g^E - \hat g^{E_0})}(a,b)\cdot c^{mn\beta} = 1\cdot \omega^{mn\beta} = 1,
	\end{equation}
	the last step because $\omega^m=1$.

	It remains to collect these vectors into a single \emph{finite} set $\Gamma$ valid for every point of $\ker\chi_h\cap Z$, as condition (2) requires. Two facts make this possible. First, the classes $E, E_0$ lie in the finite set $\mc E$, so the differences $(g^E_1 - g^{E_0}_1,\ g^E_2 - g^{E_0}_2)$ take only finitely many values. Second, the integer $\beta$ is bounded: by construction $c=\omega\in\ker\chi_h$, so $\omega^m=1$, that is, $\omega$ is an $m$-th root of unity; hence its order $d$ divides $m$, and so does its $p$-free part $\beta$. Therefore every vector produced above lies in the finite set
	\begin{equation}
		\Gamma = \set{ mn\beta\,(g_1 - g_1',\ g_2 - g_2',\ 1) :\ \beta\text{ a positive divisor of }m,\ g, g'\in F,\ (g_1, g_2)\neq (g_1', g_2') },
	\end{equation}
	each member of which is linearly independent with $h$. We have shown that every $(a,b,c)\in\ker\chi_h\cap Z$ lies in $\ker\chi_h\cap\ker\chi_v$ for some $v\in\Gamma$; that is, condition (2) holds. This completes the proof.
\end{proof}

\begin{lemma}
	\label{lemma:spectral support lemma used in case two}
	Let $F$ be a cluster in $\Z^3$ containing the origin and $p$ be a prime.
	Assume $|F|$ is a power of $p$.
	Let $h$ be an arbitrary nonzero vector in $\Z^3$. 
	Then at least one of the following must happen.
	\begin{enumerate}
		\item Every line in $\R^3$ that is parallel to $h$ intersects $F$ in a set of size divisible by $p$.
		\item There is a finite set $R$ of rational points in $\R^3/\Z^3$ and a finite set $V$ of nonzero vectors in $\Z^3$ such that 
			\begin{equation}
				\ker\chi_h
				\cap
				\lrp{
					\bigcap_{\alpha\text{ coprime to } |F|}
					Z\lrb{\sum_{g\in F} \chi_{\alpha g}}
				}
				\subseteq
				\bigcup_{\rho\in R}\bigcup_{v\in V}  \set{\rho + tv \in \R^3/\Z^3:\ t\in \T}
			\end{equation}
	\end{enumerate}
\end{lemma}
\begin{proof}
	Assume (1) does not hold.
	Then by Lemma \ref{lemma:second spectral support lemma} we can find a finite set $\Gamma\subseteq \Z^3\setminus \set{0}$ such that each element of $\Gamma$ is linearly independent with $h$ and
	\begin{equation}
		\ker\chi_h
		\cap
		\lrp{
			\bigcap_{\alpha\text{ coprime to } |F|}
			Z\lrb{\sum_{g\in F} \chi_{\alpha g}}
		}
		\subseteq
		\bigcup_{u\in \Gamma} \ker\chi_h\cap \ker\chi_u
	\end{equation}
	Using Lemma \ref{lemma:intersection of two kernels} we know that for each $u\in \Gamma$ there is a finite set of rational points $R_u$ and a nonzero vector $v_u$ such that
	\begin{equation}
		\ker\chi_h\cap \ker\chi_u
		\subseteq
		\bigcup_{\rho\in R_u}\set{\rho + tv_u\in \R^3/\Z^3:\ t\in \R}
	\end{equation}
	Define $R=\bigcup_{u\in \Gamma} R_u$ and $V = \set{v_u:\ u\in \Gamma}$.
	It follows that
	\begin{equation}
		\bigcup_{u\in \Gamma} \ker\chi_h\cap \ker\chi_u
		\subseteq 
		\bigcup_{\rho\in R}\bigcup_{v\in V} \set{\rho + tv\in \R^3/\Z^3:\ t\in \R}
	\end{equation}
	whence condition (2) immediately holds.
\end{proof}

\begin{lemma}
	\label{lemma:main technical precursor to main}
	Let $F\subseteq \Z^3$ be a cluster containing the origin and $p$ be a prime.
	Let $\rho$ be a rational point in $\R^3/\Z^3$ and $v$ be a nonzero vector in $\Z^3$ such that infinitely many points of the set
	\begin{equation}
		\set{\rho + vt\in \R^3/\Z^3:\ t\in \R}
	\end{equation}
	are in
	\begin{equation}
		\label{equation:partition equations}
		Z = \bigcap_{\alpha\text{ coprime to } p} Z\lrb{\sum_{g\in F} \chi_{\alpha g}}
	\end{equation}
	Then any plane parallel to $\set{w\in \R^3:\ \ab{w, v} = 0}$ intersects $F$ in a set of size divisible by $p$.
\end{lemma}
\begin{proof}
	Let $h_0$ be a primitive vector in $\Z^3\setminus \set{0}$ and $r$ be a rational number such that $\rho = r h_0$ in $\R^3/\Z^3$.
	Then, by hypothesis, infinitely many points of
	\begin{equation}
		\set{rh_0+ vt\in \R^3/\Z^3:\ 0\leq t< 1}
	\end{equation}
	are in $Z$.
	Let $r=c/d$, where $c$ and $d$ are relatively prime integers. 
	We get, for each $\alpha$ coprime to $p$, 
	\begin{equation}
		\sum_{g\in F} e^{2\pi i r\alpha \ab{g, h_0}}e^{2\pi i \alpha \ab{g, v} t} = 0
	\end{equation}
	for infinitely many $t\in [0, 1)$.
	Therefore, the (Laurent) polynomial
	\begin{equation}
		\label{equation:intermediate equation}
		\sum_{g\in F} e^{2\pi i r \alpha \ab{g, h_0}} z^{\alpha \ab{g, v}}
	\end{equation}
	is satisfied by infinitely many $z\in S^1$ for each $\alpha$ coprime to $p$.
	Let $d= p^m \beta$, where $m$ is a non-negative integer and $\beta$ is coprime to $p$.
	Thus, by Equation \ref{equation:intermediate equation} the polynomial
	\begin{equation}
		\sum_{g\in F} e^{2\pi i r \beta \ab{g, h_0}} z^{\beta \ab{g, v}}
		=
		\sum_{g\in F} e^{2\pi i \ab{g, h_0} c/p^m} z^{\beta \ab{g, v}}
	\end{equation}
	has infinitely many solutions in $S^1$ and is hence identically zero. 
	Define an equivalence relation $\sim_v$ on $F$ by writing $g_1\sim_v g_2$ for $g_1, g_2\in F$ if $\ab{g_1, v}=\ab{g_2, v}$.
	Let $F_1, \ldots, F_l$ be all the equivalence classes in $F$.
	The coefficients of the above polynomial are
	\begin{equation}
		\label{equation:coefficients of the polynomial}
		\sum_{g\in F_j} e^{2\pi i \ab{g, h_0} c/p^m}, \quad j=1, \ldots, l.
	\end{equation}
	and hence each of these terms is $0$.
	If $m=0$ then this cannot happen since $\sum_{g\in F_j} e^{2\pi i \ab{g, h_0}c} = |F_j|$, and hence we must have $m\geq 1$.
	This implies that $p$ divides $d$ and hence, since $c$ is relatively prime to $d$, we have $c$ is coprime to $p$.
	Let $\zeta$ denote the complex number $e^{2\pi i/p^m}$.
	By Equation \ref{equation:coefficients of the polynomial} we have $\zeta$ is a root of each of the (Laurent) polynomials
	\begin{equation}
		\sum_{g\in F_j} z^{\ab{g, h_0}c}, \quad j= 1, \ldots, l
	\end{equation}
	Therefore by Lemma \ref{lemma:cyclotomy divisibility lemma} we have that each $|F_j|$ is divisible by $p$.
	This implies that any plane parallel to $\set{w\in \R^3:\ \ab{w, v} = 0}$ intersects $F$ in a set of cardinality divisible by $p$ and we are done.
\end{proof}

\section{The Periodicity Result}
\label{section:the periodicity result}
\subsection{Main Theorem}
\label{subsection:main theorem}

\begin{theorem}
	\label{theorem:prime square theorem}
	Let $F$ be an exact cluster in $\Z^3$ with cardinality $p^2$, where $p$ is a prime.
	Then there is a $1$-weakly periodic $F$-tiling.
\end{theorem}

\subsection{Proof}

\subsubsection{Notation}
\label{subsubsection:notation in the proof of main theorem}

Let $X\subseteq \set{0, 1}^{\Z^3}$ be the set of all the $F$-tilings and $\mu$ be a $\Z^3$-ergodic probability measure on $X$, which we may assume to be concentrated on the orbit closure of a given $F$-tiling $T$.
We define $A=\set{x\in X:\ x(0) = 1}$ and $f\in L^2(X, \mu)$ as the characteristic function of $A$.
Let $\nu$ be the spectral measure associated to the unit vector $f/\norm{f}_2$ in $L^2(X, \mu)$ and $\theta$ be the unitary isomorphism between $H_f$ --- the span closure of the orbit of $f/\norm{f}_2$ --- and $L^2(\T^3, \nu)$ as discussed in Theorem \ref{theorem:spectral theorem}.
\emph{We will use $\T$ to denote the unit circle and not $\R/\Z$.}

By Lemma \ref{lemma:dilation lemma for higher level tilings} we have $T$ is an $(\alpha F)$-tiling whenever $\alpha$ is relatively prime to $p$.
Thus

\begin{equation}
	\sum_{g\in F} 1_{\alpha g A} = \sum_{g\in F} (\alpha g)\cdot 1_A = 1_X 
\end{equation}
for all $\alpha$ relatively prime to $p$.
Now $1_X$ is $\Z^3$-invariant in $L^2(X, \mu)$, and hence $\theta(1_X)$ is $\Z^3$-invariant in $L^2(\T^3, \nu)$.
But the only $\Z^3$-invariant members of $L^2(\T^3, \nu)$ are the ones in the span of $\delta_{\set{1}}$---the Dirac function concentrated at the identity of $\T^3$.
Thus, applying $\theta$, we get
\begin{equation}
	\sum_{g\in F} \chi_{\alpha g} = c \delta_{\set{1}}, \quad \forall\alpha \textup{ coprime to } p
\end{equation}
for some constant $c$.
Therefore, the spectral measure $\nu$ is supported on $Z\cup\set{1}$, where
\begin{equation}
	Z = \bigcap_{\alpha\text{ coprime to } p} Z\lrb{\sum_{g\in F} \chi_{\alpha g}}
\end{equation}
By Lemma \ref{lemma:first spectral support lemma} we know that there is a finite set $\Delta\subseteq \mathbb Z^3\setminus \set{0}$ whose elements are pairwise linearly independent such that the spectral measure $\nu$ associated to $f/\norm{f}_2$ is supported on $\bigcup_{g\in \Delta}\ker(\chi_g)$.

\subsubsection{Case 1: There exist at least two distinct members $g_0$ and $g_1$ in $\Delta$ such that every line parallel to either $g_0$ or $g_1$ intersects $F$ in a set of cardinality divisible by $p$}
By Lemma \ref{lemma:first combinatorial lemma} we deduce that $F$ is a prism with prime foundation and hence by Lemma \ref{lemma:tiling by a prism with prime base} we see that there is a $3$-periodic $F$-tiling, which is in particular $1$-weakly periodic.

\subsubsection{Case 2: There is exactly one $g_0$ in $\Delta$ such that every line parallel to $g_0$ intersects $F$ in a set of cardinality divisible by $p$.} 
In this case, we see that for each $g\in \Delta\setminus \set{g_0}$ the condition (1) in Lemma \ref{lemma:spectral support lemma used in case two} is not satisfied, and hence by Lemma \ref{lemma:spectral support lemma used in case two} condition (2) must be satisfied.
Thus for each $g\in \Delta\setminus \set{g_0}$ we can find a finite set $R_g$ of rational points in $\R^3/\Z^3$ and a finite set of vectors $V_g$ such that
\begin{equation}
	\label{equation:an equation in case two point one}
	Z\cap \ker\chi_g \subseteq \bigcup_{\rho\in R_g}\bigcup_{v\in V_g}\set{\rho+ tv\in \R^3/\Z^3:\ t\in \R}
\end{equation}
Let $R = \bigcup_{g\in \Delta\setminus\set{g_0}} R_g$ and $V = \bigcup_{g\in \Delta\setminus\set{g_0}} V_g$.
Let $N$ be a positive integer such that $\ker\chi_{Ng_0}$ contains $R$.
We further consider two subcases.

\begin{enumerate}
	\item[\qquad] \emph{Subcase 2.1:} Assume that there is $g_1$ distinct from $g_0$ in $\Delta$ such that
		\begin{equation}
			\label{equation:subcase two point one}
			(Z\cap \ker\chi_{g_1})\setminus \ker\chi_{Ng_0}
		\end{equation}
		is infinite.
		Then, in particular, $Z\cap \ker \chi_{g_1}$ is also infinite.
		Thus, by Equation \ref{equation:an equation in case two point one} there is $\rho\in R$ and $v\in V$ such that infinitely many elements of
		\begin{equation}
			\label{equation:scary equation}
			\set{\rho+vt\in \R^3/\Z^3:\ t\in \R}\setminus \ker\chi_{Ng_0}
		\end{equation}
		are in $Z$, and none of its elements are in $\ker\chi_{Ng_0}$.
		Now by Lemma \ref{lemma:main technical precursor to main} we have that whenever $\pi$ is a plane parallel to $\pi_1:=\set{w\in \R^3:\ \ab{v, w} = 0}$, we have $|\pi\cap F|$ is divisible by $p$.
		Also, by the property of $g_0$ 
		we already know that whenever $\pi$ is a plane parallel to $g_0$ we have $|\pi\cap F|$ is divisible by $p$.

		If $v$ is orthogonal to $g_0$ then, by the choice of $N$, the set in Equation \ref{equation:scary equation} is empty. 
		Thus $v$ is not orthogonal to $g_0$.
		By the fact that $v$ is not orthogonal to $g_0$, we have that a plane parallel to $g_0$ cannot be parallel to $\pi_1$.
		Then, again, by Lemma \ref{lemma:first combinatorial lemma} we deduce that $F$ is a prism with prime foundation and hence by Lemma \ref{lemma:tiling by a prism with prime base} we see that there is a $3$-periodic $F$-tiling, which is in particular $1$-weakly periodic.\vspace{0.5em}

	\item[\qquad] \emph{Subcase 2.2:}
		Now assume that $(Z\cap \ker\chi_{h})\setminus \ker\chi_{Ng_0}$ is finite whenever $h\in \Delta\setminus \set{g_0}$.
		For each $h\in \Delta\setminus\set{g_0}$, let $S_h=(Z\cap \ker\chi_h)\setminus \ker\chi_{Ng_0}$.
		Then, by our assumption, each $S_h$ is finite.
		Since $\supp(\nu)$ is contained in $Z\cap \bigcup_{g\in \Delta} \ker\chi_{g}$, we see that, in particular,
		\begin{equation}
			\text{supp}(\nu)
			\subseteq
			\ker\chi_{Ng_0} \cup \lrp{\bigcup_{h\in \Delta\setminus \set{g_0}} Z\cap \ker\chi_h}
			=
			\ker\chi_{Ng_0} \cup\lrp{ \bigcup_{h\in \Delta\setminus \set{g_0}} S_h}
		\end{equation}
		Therefore
		\begin{equation}
			\ker(\chi_{Ng_0}) \cup \bigcup_{h\in \Delta:\ h\neq g_0} S_h
		\end{equation}
		is a $\nu$-full measure set.

		Replacing $Ng_0$ by $g_0$ for simplicity of notation, we infer that there is a nonzero vector $g_0$ in $\Z^3$ and a finite set $S\subseteq \T^3$ disjoint with $\ker\chi_{g_0}$ such that $\nu$ is supported on $\ker(\chi_{g_0})\cup S$.
		We may assume that $S$ has smallest size with this property and hence each element of $S$ has positive mass under $\nu$.
		The minimality of the size of $S$ also implies that $\chi_{g_0}(s)$ is irrational for all $s\in S$, for otherwise we could replace $g_0$ by a scale of itself and reduce the size of $S$.

		We will show that $S$ is empty.
		Assume on the contrary that $S$ is non-empty.
		Define an equivalence relation $\sim$ on $S$ by writing $p\sim q$ for $p, q\in S$ if $\chi_{g_0}(p)=\chi_{g_0}(q)$.
		Let $\mc E$ be the set of all the equivalence classes.
		Thus
		\begin{equation}
			1 = 1_{\ker(\chi_{g_0})} + \sum_{E\in \mc E} 1_E
		\end{equation}
		in $L^2(\T^3, \nu)$.
		Choose a representative $p_E\in E$ for each $E\in \mc E$.
		Now acting both sides of the above equation by $ng_0$, where $n$ is any non-negative integer, we get
		\begin{equation}
			\chi_{n g_0} = 1_{\ker(\chi_{g_0})} + \sum_{E\in \mc E}\chi_{g_0}(p_E)^n 1_E
		\end{equation}
		Going back into the $L^2(X, \mu)$ world by applying $\theta^{-1}$, we get that
		\begin{equation}
			\label{equation:hook super equation}
			(n g_0)\cdot f = f^{g_0} + \sum_{E\in \mc E} \chi_{g_0}(p_E)^n \vp_E
		\end{equation}
		where $\vp_E = \norm{f}_2 \theta^{-1}(1_E)$, and $f^{g_0}$ is the orthogonal projection of $f$ onto the space of $g_0$-invariant functions in $L^2(X, \mu)$.\footnote{By the Birkhoff ergodic theorem we see that $f^{g_0}$ lies in $H_f$.
			Also, $1_{\ker\chi_{g_0}}$ is the orthogonal projection of $1$ onto the space of $g_0$-invariant functions in $L^2(\T^3, \nu)$. The fact that $\theta$ is a unitary isomorphism shows that $\theta^{-1}(1_{\ker\chi_{g_0}})$ is the same as $f^{g_0}/\norm{f}_2$.}
		Therefore, for each $n\geq 0$, we have
		\begin{equation}
			(ng_0)\cdot f - f
			=
			\lrp{f^{g_0} + \sum_{E\in \mc E} \chi_{g_0}(p_E)^n \vp_E} - \lrp{f^{g_0} - \sum_{E\in \mc E} \vp_E}
			=
			\sum_{E\in \mc E} (\chi_{g_0}(p_E)^n - 1)\vp_E
		\end{equation}
		Let $Y\subseteq X$ be a $\mu$-full measure subset of $X$ such that
		\begin{equation}
			[(ng_0)\cdot f](y) - f(y) = \sum_{E\in \mc E}(\chi_{g_0}(p_E)^n - 1) \vp_E(y)
		\end{equation}
		for all $y\in Y$ and all $n\geq 0$.
		Therefore
		\begin{equation}
			\sum_{E\in \mc E} (\chi_{g_0}(p_E)^n - 1) \vp_E(y) \in \Z
		\end{equation}
		for all $y\in Y$, and all $n\geq 0$.
		But since each $\chi_{g_0}(p_E)$ is irrational, we may apply Lemma \ref{lemma:non measure preparatory lemma} to deduce that for any $y\in Y$ the only value the above expression can take, for any non-negative integer $n$, and hence in particular for $n=1$, is $0$.
		Therefore $[g_0 \cdot f] (y) - f(y)$ is $0$ for each $y\in Y$.
		But since $Y$ is a full measure set in $X$, we infer that $g_0\cdot f = f$ in $L^2(X, \mu)$.
		Thus $f$ is $1$-periodic, and hence by Lemma \ref{label:bhattacharya correspondence principle} we deduce that the orbit closure of $T$ has a $1$-periodic point in it, which is in particular $1$-weakly periodic.

\end{enumerate}
\subsubsection{Case 3: There is no $g$ in $\Delta$ such that every line parallel to $g$ intersects $F$ in a set of size divisible by $p$} 
\label{subsubsection:last case in the proof}

In this case, by Lemma \ref{lemma:spectral support lemma used in case two} we deduce that there exist nonzero vectors $v_1, \ldots, v_n\in \Z^3$ and finite subsets $S_1, \ldots, S_n\subseteq \R^3/\Z^3$ such that each member of each $S_i$ is a rational point and the measure $\nu$ is supported on\footnote{This is because for linearly independent vectors $g$ and $h$ in $\Z^3$ we have $\ker\chi_g\cap \ker\chi_h$ is equal to $S+\set{tv\in \R^3/\Z^3:\ t\in \R}$ for some nonzero $v\in \Z^3$ orthogonal to both $g$ and $h$ and some finite set $S$ of rational points in $\R^3/\Z^3$.}
\begin{equation}
	\bigcup_{i=1}^n (S_i+ \set{t v_i\in \R^3/\Z^3:\ t\in\R})
\end{equation}
We are now done by Theorem \ref{theorem:bhattacharyas proof in three dimensions}.

%
%
%
%

%

\appendix

\section{Algebra}
\label{section:algebra appendix}

\begin{definition}
	A nonzero vector $a = (a_1, \ldots, a_n)$ in $\Z^n$ is said to be \define{primitive} if $$\gcd(a_1, \ldots, a_n) = 1$$
\end{definition}

\begin{lemma}
	\label{lemma:glnz acts transitively on the set of all the primitive vectors}
	Let $n\geq 1$.
	Then $GL_n(\Z)$ acts transitively on the set of all the primitive vectors in $\Z^n$.
\end{lemma}
\begin{proof}
	Let $a\in \Z^n$ be an arbitrary primitive vector.
	We will show that there is $T\in GL_n(\Z)$ such that $Te_n = a$, where $e_n = (0, \ldots, 0, 1)$.
	Let $\Lambda$ be the subgroup of $\Z^n$ generated by $a$.
	We claim that $\Z^n/\Lambda$ has no torsion.
	Assume on the contrary that there is torsion in $\Z^n/\Lambda$, so that there is a non-trivial element $g\in \Z^n/\Lambda$ and a positive integer $k$ such that $k\cdot g = 0$.
	This implies the existence of an element $v\in \Z^n\setminus \Lambda$ such that $kv\in \Lambda$.
	Thus $kv = ma$ for some nonzero integer $m$.
	It follows that $k$ must divide $m$, for otherwise $a$ would not be primitive.
	But then $v\in \Lambda$, contrary to the choice of $v$.

	So $\Z^n/\Lambda$ has no torsion, and thus it is isomorphic to $\Z^{n-1}$.
	Let $v_1, \ldots, v_{n-1}$ in $\Z^n$ be such that $v_1+\Lambda, \ldots, v_{n-1} + \Lambda$ forms a basis of $\Z^n/\Lambda$.
	Declare $v_n = a$ and we get a basis $v_1, \ldots, v_n$ of $\Z^n$.
	Now the $\Z$-linear map $T:\Z^n\to \Z^n$ which takes $e_i$ to $v_i$, where $e_i$ is the $i$-th standard basis vector, is an isomorphism and hence an element of $GL_n(\Z)$.
	By definition $Te_n = a$.
\end{proof}

\begin{lemma}
	\label{lemma:flattening subgroup lemma}
	Let $n$ be a positive integer and $\Lambda$ be a rank-$k$ subgroup of $\Z^n$.
	Then there is $T\in GL_n(\Z)$ such that $T(\Lambda)$ is contained in $\Z^k\times \set{(0, \ldots, 0)}$.
\end{lemma}
\begin{proof}
	Let $N\geq 1$ be an integer such that there is an isomorphism $\vp:\Z^n/\Lambda\to \Z^{n-k}\times G$, where $G$ is a finite abelian group.
	Let $\Gamma = (\vp\circ \pi)^{-1}(\set{0}\times G)$, where $\pi$ is the natural projection $\Z^n\to \Z^n/\Lambda$.
	Then $\Gamma$ contains $\Lambda$ and $\Z^n/\Gamma$ is isomorphic to $\Z^{n-k}$.
	We can now choose vectors $v_1, \ldots, v_{n-k}$ in $\Z^n$ such that $v_1 + \Gamma, \ldots, v_{n-k}+\Gamma$ is a basis of $\Z^n/\Gamma$.
	If $u_1, \ldots, u_k$ is a basis of $\Gamma$, then $u_1, \ldots, u_k, v_1, \ldots, v_{n-k}$ forms a basis of $\Z^n$.
	Define a $\Z$-linear map $T:\Z^n\to \Z^n$ by declaring $Te_i = u_i$ for $1\leq i\leq k$ and $Te_i = v_{i - k}$ for $k+1\leq i\leq n$.
	Then $T$ is a surjective $\Z$-linear map and hence is a member of $GL_n(\Z)$.
	Now $T^{-1}$ is an element of $GL_n(\Z)$ which takes $\Gamma$ to $\Z^k\times \set{(0), \ldots, 0}$, and hence puts $\Lambda$ inside $\Z^k\times \set{(0, \ldots, 0)}$, finishing the proof.
\end{proof}

\section{Measure Theory}

\begin{lemma}
	\label{lemma:weak star measurability}
	Let $(X, \mc B_X)$ be a measurable space, let $Z$ be a topological space, let $Y$ be a set, and let $\mc F$ be a family of functions from $Y$ to $Z$.
	Equip $Y$ with the $\sigma$-algebra $\mc B_Y$ generated by the sets $f^{-1}(W)$ with $f\in \mc F$ and $W\subseteq Z$ open.
	Equivalently, $\mc B_Y$ is the smallest $\sigma$-algebra on $Y$ making every $f\in \mc F$ Borel measurable.
	Then a map $\vp:X\to Y$ is measurable if and only if $f\circ \vp:X\to Z$ is measurable for every $f\in \mc F$.
\end{lemma}
\begin{proof}
	Suppose first that $\vp$ is measurable.
	Each $f\in\mc F$ is $\mc B_Y$-measurable by the very definition of $\mc B_Y$, so $f\circ\vp$ is a composition of measurable maps and is therefore measurable.

	Conversely, suppose $f\circ \vp$ is measurable for every $f\in \mc F$.
	Let
	\[
		\mc G = \set{f^{-1}(W):\ f\in \mc F,\ W\subseteq Z \text{ open}},
	\]
	so that $\mc B_Y = \sigma(\mc G)$.
	For $f^{-1}(W)\in \mc G$ we have
	\[
		\vp^{-1}\!\lrp{f^{-1}(W)} = (f\circ \vp)^{-1}(W)\in \mc B_X,
	\]
	since $f\circ \vp$ is measurable and $W$ is open.
	Now the collection $\mc A = \set{B\subseteq Y:\ \vp^{-1}(B)\in \mc B_X}$ is a $\sigma$-algebra, because taking preimages commutes with complements and with countable unions.
	We have just shown $\mc G\subseteq \mc A$, so $\mc A\supseteq \sigma(\mc G) = \mc B_Y$.
	Thus $\vp^{-1}(B)\in \mc B_X$ for every $B\in \mc B_Y$, that is, $\vp$ is measurable.
\end{proof}

The sets $f^{-1}(W)$ form a subbasis for the initial topology on $Y$ induced by $\mc F$, so $\mc B_Y$ is the $\sigma$-algebra generated by that subbasis.
When the initial topology is second countable, $\mc B_Y$ coincides with the Borel $\sigma$-algebra of $Y$; this is the case in our application, where $Y$ is the space $\mc P(\R/\Z)$ of probability measures on $\R/\Z$ with the (compact, metrizable) weak-$*$ topology and $\mc F$ is the family of integration functionals $\mu\mapsto \int \xi\, d\mu$, $\xi\in C(\R/\Z)$.

\begin{lemma}
	\label{lemma:sequences converging to zero form a measurable set}
	Let $S\subseteq \R^{\N}$ be the set of all the points $x=(x_n:\ n\in \N)$ in $\R^{\N}$ such that $x_n\to 0$.
	Then $S$ is a measurable set.
\end{lemma}
\begin{proof}
	For each $q\geq 1$ and each $k\geq 1$, define
	\begin{equation}
		E_{q, k} = \set{(x_n:\ n\in \N):\ -1/k<x_q< 1/k}
	\end{equation}
	Then each $E_{q, k}$ is measurable and it is easy to check that
	$$
	S=\bigcap_{k\geq1} \bigcup_{N\geq 1} \bigcap_{q\geq N} E_{q, k}
	$$
	whence $S$ is a measurable set.
\end{proof}
\section{Weak Periodicity Assuming the Spectral Measure is Supported on the Union of Finitely Many Lines}
\label{section:bhattachayas proof in three dimension}

The goal of this section is to prove the following.

\begin{theorem}
	\label{theorem:bhattacharyas proof in three dimensions}
	Let $F\subseteq\Z^3$ be an exact cluster and $T\subseteq\Z^3$ be an $F$-tiling.
	Let $X\subseteq \set{0, 1}^{\Z^3}$ be the orbit closure of $T$ and $\mu$ be a $\Z^3$-ergodic probability measure on $X$.
	Define
	$$
	A=\set{x\in X:\ x(0) = 1}
	$$ and $f\in L^2(X, \mu)$ as the characteristic function of $A$.
	Let $\nu$ be the spectral measure associated to the unit vector $f/\norm{f}_2$ in $L^2(X, \mu)$ and $\theta$ be the unitary isomorphism between $H_f$ --- the span closure of the orbit of $f/\norm{f}_2$ --- and $L^2(\T^3, \nu)$ as discussed in Theorem \ref{theorem:spectral theorem}.
	Assume that there exist nonzero vectors $v_1, \ldots, v_n\in \Z^3$ and finite subsets $S_1, \ldots, S_n\subseteq \T^3 = \R^3/\Z^3$ such that each member of each $S_i$ is a rational point and the measure $\nu$ is supported on
	\begin{equation}
		\bigcup_{i=1}^n (S_i+ \set{t v_i\in \T^3:\ t\in\R})
	\end{equation}
	Then $A$ is $1$-weakly periodic, and hence there is a $1$-weakly periodic $F$-tiling.
\end{theorem}


The proof is an adaptation of Bhattacharya's proof of the periodic tiling conjecture (for $\Z^2$) in \cite{bhattacharya_tilings}.
No fundamentally new ideas are needed.
However, since Theorem \ref{theorem:bhattacharyas proof in three dimensions} is not a \emph{direct} corollary of the work done in \cite{bhattacharya_tilings}, we give full details.
Wherever possible, we give reference to corresponding results in \cite{bhattacharya_tilings}.
Throughout this section we will use the notation in Theorem \ref{theorem:bhattacharyas proof in three dimensions}.
We now begin the proof.

We may assume that $n$ is the smallest integer for which one can find vectors $v_1, \ldots, v_n\in \Z^3$ and finite subsets $S_1, \ldots, S_n\subseteq \T^3$ such that the support of $\nu$ is contained in
\begin{equation}
	\bigcup_{i=1}^n (S_i + \set{tv_i\in \T^3:\ t\in \R})
\end{equation}
Then $v_1, \ldots, v_n$ are pairwise linearly independent. 

\begin{lemma}
	\label{lemma:an elementary kernel containment lemma used in last section}
	Let $S\subseteq\T^3 = \R^3/\Z^3$ be a finite set of rational points and $v$ be an arbitrary nonzero vector in $\Z^3$.
	Then there exist linearly independent vectors $g$ and $h\in \Z^3$ such that both $g$ and $h$ are orthogonal to $v$ and
	\begin{equation}
		S + \set{tv\in \R^3/\Z^3:\ t\in \R}
	\end{equation}
	is contained in $\ker\chi_g\cap \ker\chi_h$.
\end{lemma}
\begin{proof}
	The set of all vectors in $\Z^3$ which are orthogonal to $v$ forms a rank-$2$ subgroup of $\Z^3$.
	Thus we can find two linearly independent vectors $g_0$ and $h_0$ in $\Z^3$ which are both orthogonal to $v$.
	Thus
	\begin{equation}
		\set{tv\in \R^3/\Z^3:\ t\in\R} \subseteq \ker\chi_{g_0}\cap \ker\chi_{h_0}
	\end{equation}
	If $N$ is a positive integer such that $Ns = 0$ in $\R^3/\Z^3$ for each $s\in S$, then we see that
	\begin{equation}
		S + \set{tv\in \R^3/\Z^3:\ t\in\R} \subseteq \ker\chi_{Ng_0}\cap \ker\chi_{Nh_0}
	\end{equation}
	Thus $g = Ng_0$ and $h=Nh_0$ satisfy the requirement of the lemma.
\end{proof}

\begin{lemma}
	\label{lemma:choosing planes intersection and all that}
	We can choose vectors $g_1, h_1, \ldots, g_n, h_n\in \Z^3$ such that
	\begin{enumerate}[a)]
		\item $S_i+\set{tv_i\in \R^3/\Z^3:\ t\in \R} \subseteq \ker\chi_{g_i}\cap \ker\chi_{h_i}$ for each $i$.
		\item $\Z g_i+ \Z h_i$ is a rank-$2$ subgroup of $\Z^3$ for each $i$.
		\item $g_i$ and $h_i$ are orthogonal to $v_i$ for each $i$.
		\item $(\Z g_i + \Z h_i) + (\Z g_j + \Z h_j)$ is a rank-$3$ subgroup of $\Z^3$ whenever $i\neq j$.
	\end{enumerate}
\end{lemma}
\begin{proof}
	By Lemma \ref{lemma:an elementary kernel containment lemma used in last section} we can find $g_1, h_1, \ldots, g_n, h_n$ such that (a), (b) and (c) are true.
	We show that (d) also holds.
	Let $i, j\in \set{1, \ldots, n}$ be distinct.
	Assume on the contrary that
	\begin{equation}
		\Lambda := (\Z g_i + \Z h_i) + (\Z g_j + \Z h_j)
	\end{equation}
	is not a rank-$3$ subgroup of $\Z^3$.
	Then $\Lambda$ is a rank-$2$ subgroup of $\Z^3$.
	Consequently, both $v_i$ and $v_j$ are orthogonal to each of $g_i, h_i, g_j, h_j$.
	This forces that $v_i$ and $v_j$ are linearly dependent, contradicting the pairwise linear independence of $v_1, \ldots, v_n$ noted above (a consequence of the minimality of $n$).
\end{proof}

Fix $g_1, h_1, \ldots, g_n, h_n$ as in Lemma \ref{lemma:choosing planes intersection and all that}.
Thus, in particular, the measure $\nu$ is supported on the set
\begin{equation}
	\bigcup_{i=1}^n \ker\chi_{g_i} \cap \ker\chi_{h_i}
\end{equation}
Let $\Lambda_i$ be the subgroup of $\Z^3$ generated by $g_i$ and $h_i$, that is $\Lambda_i=\Z g_i + \Z h_i$.

\subsection{Very Weak Periodic Decomposition}

Let $H_f = \overline{\text{Span}}\set{v\cdot f:\ v\in \Z^3}$.
Let $\nu$ be the spectral measure associated to $f/\norm{f}_2$ and $\theta:H_f\to L^2(\T^3, \nu)$ be the unitary isomorphism discussed in Section \ref{section:spectral theorem}.

For any subgroup $\Lambda$ of $\Z^3$ let $f^\Lambda$ denote the projection of $f$ onto the subspace of all the $\Lambda$-invariant functions in $L^2(X, \mu)$.
Writing $B_N$ to denote the ball of radius $N$ in $\Z^3$, by the ergodic theorem we have
\begin{equation}
	f^\Lambda
	=
	\lim_{N\to \infty} \frac{1}{|\Lambda\cap B_N|}\sum_{v\in \Lambda\cap B_N} v\cdot f
\end{equation}
in $L^2(X, \mu)$. 
Therefore $f^\Lambda$ lies in $H_f = \overline{\text{Span}}\set{v\cdot f:\ v\in \Z^3}$, and it is hence the same as the image of $f$ under the  orthogonal projection $H_f\to H_f$ onto the set of all the $\Lambda$-invariant functions in $H_f$.
This implies that $\theta(f^\Lambda/\norm{f}_2)= \one^\Lambda$, where $\one^\Lambda$ is the orthogonal projection of $\one$ onto the space of all the $\Lambda$-invariant functions in $L^2(\T^3, \nu)$. 

\begin{lemma}
	\label{lemma:projection of the constant function one}
	Let $\rho$ be a probability measure on $\T^3$.
	Let $u$ and $w$ be two linearly independent vectors in $\Z^3$ and $\Lambda$ be the subgroup of $\Z^3$ generated by $u$ and $w$.
	Under the natural unitary action of $\Z^3$ on $L^2(\T^3, \rho)$, we have
	\begin{equation}
		\one^\Lambda
		=
		1_{\ker\chi_u\cap \ker\chi_w}
	\end{equation}
	where $\one^\Lambda$ denotes the orthogonal projection of $\one$ onto the space of all the $\Lambda$-invariant functions in $L^2(\T^3, \rho)$.
\end{lemma}
\begin{proof}
	Write $K = \ker\chi_u\cap \ker\chi_w$.
	Recall that the generators act on $L^2(\T^3, \rho)$ by multiplication by their characters: $u$ acts as the operator $\sigma_u\colon\phi\mapsto\chi_u\phi$, and $w$ as $\sigma_w\colon\phi\mapsto\chi_w\phi$.
	Since $\Lambda$ is generated by $u$ and $w$, a function $\psi\in L^2(\T^3, \rho)$ is $\Lambda$-invariant if and only if $\sigma_u\psi = \psi$ and $\sigma_w\psi = \psi$, that is, the pointwise products satisfy $\chi_u\psi = \psi$ and $\chi_w\psi = \psi$ $\rho$-a.e.
	As $|\chi_u(x)| = 1$ for every $x$, the equation $\chi_u(x)\psi(x) = \psi(x)$ forces $\psi(x) = 0$ wherever $\chi_u(x)\neq 1$.
	Thus $\psi$ vanishes $\rho$-a.e. away from $\ker\chi_u$, and likewise away from $\ker\chi_w$, hence off $K$.
	
	Thus the $\Lambda$-invariant functions form exactly the closed subspace $V$ of functions supported on $K$.
	For any $\phi\in L^2(\T^3, \rho)$ the decomposition $\phi = 1_K\phi + (1-1_K)\phi$ splits $\phi$ into a function supported on $K$ and one vanishing on $K$.
	The latter is orthogonal to every member of $V$, so the orthogonal projection onto $V$ is multiplication by $1_K$.
	Applying this to $\one$ gives $\one^\Lambda = 1_K\cdot\one = 1_K = 1_{\ker\chi_u\cap \ker\chi_w}$.
\end{proof}

\begin{lemma}
	\label{lemma:very weakly periodic decomposition theorem}
	\emph{(See \cite[Theorem 3.3]{bhattacharya_tilings})}
	The function
	$$
	f_0 := f-\sum_{i=1}^nf^{\Lambda_i}
	$$
	is $3$-periodic, where recall that $\Lambda_i = \Z g_i + \Z h_i$.
\end{lemma}
\begin{proof}
	The map $\theta$ is linear and, by Theorem \ref{theorem:spectral theorem}, and satisfies $\theta(f/\norm{f}_2)=\one$.
	Moreover $\theta(f^{\Lambda_i}/\norm{f}_2)=\one^{\Lambda_i}$ for each $i$, by the relation $\theta(f^\Lambda/\norm{f}_2)=\one^\Lambda$ established above applied with $\Lambda=\Lambda_i$.
	Hence
	\begin{equation}
		\frac{1}{\norm{f}_2} \theta\lrp{f - \sum_{i=1}^n f^{\Lambda_i}}
		=
		\theta\lrp{\frac{f}{\norm{f}_2}} - \sum_{i=1}^n \theta\lrp{\frac{f^{\Lambda_i}}{\norm{f}_2}}
		=
		\one - \sum_{i=1}^n\one^{\Lambda_i} .
	\end{equation}
	By Lemma \ref{lemma:projection of the constant function one} we have $\one^{\Lambda_i} = 1_{\ker\chi_{g_i}\cap \ker\chi_{h_i}}$ since $g_i$ and $h_i$ generate $\Lambda_i$.
	So we just need to show that $\one - \sum_{i=1}^n\one^{\Lambda_i}$ is $3$-periodic.
	Note that since $\nu$ is supported on $\bigcup_{i=1}^n\ker\chi_{g_i}\cap \ker\chi_{h_i}$, we have
	\begin{equation}
		\one = 1_{\bigcup_{i=1}^n \ker\chi_{g_i}\cap \ker\chi_{h_i}}
	\end{equation}
	in $L^2(\T^3, \nu)$.
	So we need to show that the function
	\begin{equation}
		D :=
		1_{\bigcup_{i=1}^n \ker\chi_{g_i}\cap \ker\chi_{h_i}}
		-
		\sum_{i=1}^n 1_{\ker\chi_{g_i}\cap \ker\chi_{h_i}}
	\end{equation}
	is $3$-periodic.
	Write $K_i = \ker\chi_{g_i}\cap\ker\chi_{h_i}$.
	At a point lying in exactly one of the $K_i$ the two terms of $D$ agree, so $D$ vanishes there.
	Hence $D$ is supported on the union of the pairwise intersections $\bigcup_{i\neq j}(K_i\cap K_j)$.
	(When $n=1$ this union is empty, so $D=0$ and the claim is trivial.
	Assume $n\geq 2$ below.)
	Now consider the finite index subgroup
	\begin{equation}
		\Gamma
		=
		\bigcap_{i\neq j}(\Lambda_i+\Lambda_j)
	\end{equation}
	of $\Z^3$, which is of finite index because each $\Lambda_i+\Lambda_j$ has rank $3$ by property (d) of Lemma \ref{lemma:choosing planes intersection and all that}.
	If $t\in K_i\cap K_j$, then 
	$$
		\chi_{g_i}(t)=\chi_{h_i}(t)=\chi_{g_j}(t)=\chi_{h_j}(t)=1,
	$$
	so $\chi_w(t)=1$ for every $w\in\Lambda_i+\Lambda_j$, and in particular for every $w\in\Gamma\subseteq\Lambda_i+\Lambda_j$.
	Therefore $\chi_w D = D$ for every $w\in\Gamma$, that is, $D$ is invariant under the rank-$3$ subgroup $\Gamma$ and so is $3$-periodic.
\end{proof}

\begin{lemma}
	\label{lemma:a special periodicity lemma}
	\emph{(See \cite[Theorem 3.3]{bhattacharya_tilings})}
	For any $l\geq 1$ and $i\in \set{1, \ldots, n}$ we have $f^{l\Lambda_i}-f^{\Lambda_i}$ is $3$-periodic.
\end{lemma}
\begin{proof}
	We have
	$$
	\theta(f^{l\Lambda_i} - f^{\Lambda_i}) = \norm{f}_2 (1_{\ker\chi_{l g_i}\cap \ker\chi_{l h_i}} - 1_{\ker\chi_{g_i}\cap \ker\chi_{h_i}})
	$$
	So it suffices to show that
	$$
	1_{\ker\chi_{l g_i}\cap \ker\chi_{l h_i}} - 1_{\ker\chi_{g_i}\cap \ker\chi_{h_i}}
	$$
	is $3$-periodic in $L^2(\T^3, \nu)$.
	This is a simple exercise keeping in mind that $\nu$ is supported on $\bigcup_{i=1}^n \ker\chi_{g_i}\cap \ker\chi_{h_i}$. 
\end{proof}

%
%
%

%
\subsection{Polynomial Sequences}
\label{subsection:polynomial sequences}

By a \define{bi-infinite sequence} in $\R/\Z$, we mean a map $s:\Z\to \R/\Z$.
We will write $s_i$ to mean $s(i)$.
In what follows we will write `sequence' to mean a bi-infinite sequence.
For a sequence $s$, we define another sequence $\partial s$ by writing $(\partial s)_i= s_{i+1}-s_i$.
The $k$-fold composition $\partial \circ \cdots \circ \partial$ will be denoted by $\partial^k$.
We say that a sequence $s$ is a \define{polynomial sequence} if $\partial^ks=0$ for some positive integer $k$.
It is easy to argue by induction that if $s$ is a polynomial sequence in $\R/\Z$, then there exists a non-negative integer $n$ and $a_0, a_1, \ldots, a_n\in \R/\Z$ such that $s_k = a_0 + a_1k + \cdots + a_nk^n$ for all $k$, where the terms $a_ik^i$ have to be interpreted as elements in $\R/\Z$ by thinking of $a_ik^i$ as the $k^i$-fold sum of $a_i$.
This justifies the usage of the phrase `polynomial sequence.'

There is a natural action of $\Z[u, u^{-1}]$, the ring of all the Laurent polynomials over $\Z$, on the set of all the sequences in $\R/\Z$.
First note that there is a natural way to add two sequences: For $s, t:\Z\to \R/\Z$, we have $s+t:\Z\to \R/\Z$ defined as $(s+t)_i = s_i+ t_i$.
Similarly, any sequence can be scaled by an integer.
Now for $n\in \Z$ and a sequence $s:\Z\to \R/\Z$, we define the sequence $u^ns:\Z\to \R/\Z$ as $(u^ns)_i= s_{i+n}$.
For an element $\sum_{n\in S} a_nu^n$ in $\Z[u, u^{-1}]$, and a sequence $s:\Z\to \R/\Z$, we define
\begin{equation}
	\lrp{\sum_{n\in S} a_nu^n}s = \sum_{n\in S} a_n (u^ns)
\end{equation}

\begin{lemma}
	\label{lemma:unipotent polynomial on cosets}
	Let $s:\Z\to \R/\Z$ be a sequence such that there is $m\geq 1$ and $k\geq 1$ with the property that $(u^m-1)^ks=0$.
	Then the restriction of $s$ to any coset of $m\Z$ in $\Z$ is a polynomial sequence.
\end{lemma}
\begin{proof}
	We induct on $k$.
	For $k=1$ the hypothesis $(u^m-1)s=0$ reads $s_{i+m}=s_i$ for all $i$, so $s$ is constant on each coset of $m\Z$, which is a polynomial sequence (of degree $0$).
	Fix $k>1$ and inductively assume that the proposition has been proved for all smaller values.
	Define $t=(u^m-1)s$.
	Now since $(u^m-1)^{k-1}t=0$, we deduce by the inductive hypothesis that $t$ restricted to any coset of $m\Z$ is a polynomial sequence.
	Let $a_0, \ldots, a_n$ be elements of $\R/\Z$ such that
	\begin{equation}
		\label{equation:polynomial expression}
		t_{mj} = a_0+a_1 j + \cdots + a_nj^n
	\end{equation}
	for all $j$.
	Now since $t = (u^m-1)s$, we have
	\begin{align}
		\begin{split}
			s_{m+j} - s_j &= t_j\\
			\ra s_{m+mj} - s_{mj} &= t_{mj} \\
			\ra s_{m+m(N-1)} - s_0 &= \sum_{j=0}^{N-1} t_{mj}\\
			\ra s_{mN} &= s_0 + \sum_{j=0}^{N-1} t_{mj}\\
		\end{split}
	\end{align}
	the third line being the telescoping sum of the second over $j=0, \ldots, N-1$ (valid for $N\geq 0$).
	More conceptually, the bi-infinite sequence $a_N := s_{mN}$ has first difference $a_{N+1}-a_N = s_{m(N+1)}-s_{mN} = t_{mN}$ for all $N\in\Z$, which by Equation \ref{equation:polynomial expression} is a polynomial in $N$; and a bi-infinite sequence in $\R/\Z$ whose first difference is a polynomial sequence is itself a polynomial sequence (apply $\partial$ once more to see $\partial^{n+2}a=0$).
	Hence $N\mapsto s_{mN}$ is a polynomial sequence, that is, $s$ restricted to $m\Z$ is a polynomial sequence.
	The same argument applied with base point $r$ in place of $0$, using that $t$ restricted to the coset $r+m\Z$ is a polynomial sequence, shows that $s$ restricted to each coset $r+m\Z$ is a polynomial sequence, completing the induction.
\end{proof}

\begin{corollary}
	\label{corollary:annihilated by a polynomial unipotent}
	Let $s$ be a sequence and $m_1, \ldots, m_l$ be positive integers such that
	\begin{equation}
		(u^{m_1}-1) \cdots (u^{m_l}-1) s=0
	\end{equation}
	Then there are positive integers $m$ and $k$ such that $(u^m-1)^k$ annihilates $s$, and thus $s$ is a polynomial sequence on each coset of $m\Z$.
\end{corollary}
\begin{proof}
	The roots of the polynomial $u^{m_i}-1$ are roots of unity.
	Therefore the polynomial $(u^{m_1}-1) \cdots (u^{m_l}-1)$ divides $(u^m-1)^k$ for a suitable choice of positive integers $m$ and $k$.
	Now the desired result follows by the previous lemma.
\end{proof}

%
%
%

%
\subsection{Equidistribution}
\label{subsection:equidistribution}

Let $\lambda$ denote the Haar measure on $\R/\Z$.
For any sequence $s:\Z\to \R/\Z$, we define a probability measure $\lambda_n$ on $\R/\Z$ as
\begin{equation}
	\lambda_n = \frac{1}{2n+1} \sum_{i=-n}^n \delta_{s_i}
\end{equation}
where $\delta_{s_i}$ denotes the Dirac mass at $s_i$.
We say that $s$ is \define{equidistributed} if $\lambda_n\to \lambda$ in the weak-$*$ sense.
Weyl's equidistribution theorem\footnote{See Theorem 1.4 in \cite{einsiedler_ward_ergodic_theory}} states that every polynomial sequence in $\R/\Z$ is either periodic or equidistributes.

Similarly, we say that a \emph{$k$-dimensional sequence} $t:\Z^k \to \R/\Z$ is \define{equidistributed} if the probability measures
\begin{equation}
	\lambda_n = \frac{1}{(2n+1)^k} \sum_{p\in [-n,n]^k} \delta_{t_p}
\end{equation}
converges to $\lambda$ in the weak-$*$ sense.
For later use we need the following lemma.

\begin{lemma}
	\label{lemma:the equidistribution points form a measurable set}
	Let $(Y, \mc Y, \nu)$ be a probability space equipped with a measure preserving $\Z^3$-action.
	Let $\Gamma$ be a finite index subgroup of $\Z^3$ and $\vp:Y\to \R/\Z$ be a measurable function.
	For each $y\in Y$ we have the function $\vp_y:\Gamma\to \R/\Z$ which takes $g\in \Gamma$ to $\vp(g\cdot y)$.
	Then the sets
	\begin{equation}
		E=\set{y\in Y:\ \vp_y:\Gamma\to \R/\Z \text{ equidistributes}}, \quad
		K=\set{y\in Y:\ \vp_y:\Gamma\to \R/\Z \text{ is constant}}
	\end{equation}
	are measurable.
\end{lemma}
\begin{proof}\footnote{Thanks to Ronnie Pavlov and Nishant Chandgotia for helping with this proof.}
	We give the proof for $\Gamma=\Z^3$.
	The general case is identical after identifying $\Gamma$ with $\Z^3$ via a choice of basis, the resulting $\Z^3$-action on $(Y, \mc Y, \nu)$ being again measure preserving.
	First we show that $E$ is measurable.
	We know that $\mc P(\R/\Z)$, the set of all the probability measures in $\R/\Z$, is metrizable (in the weak-$*$ topology).
	Choose a metric $d$ on $\mc P(\R/\Z)$.
	For each $n\geq 1$ and each $y\in Y$, define probability measure $\lambda_{n, y}$ by writing
	\begin{equation}
		\lambda_{n, y} = \frac{1}{(2n+1)^3} \sum_{p\in [-n, n]^3} \delta_{\vp_y(p)}
	\end{equation}
	For a fixed continuous $\xi\in C(\R/\Z)$, the composition of $y\mapsto\lambda_{n, y}$ with the integration functional $\eta\mapsto\int \xi\, d\eta$ is
	\[
		y\longmapsto \int \xi\, d\lambda_{n, y} = \frac{1}{(2n+1)^3}\sum_{p\in[-n, n]^3} \xi(\vp(p\cdot y)),
	\]
	a finite sum of measurable functions (each $y\mapsto\vp(p\cdot y)$ is measurable since the action is measurable and $\vp$ is measurable), hence measurable.
	As the integration functionals generate the weak-$*$ structure on $\mc P(\R/\Z)$, Lemma \ref{lemma:weak star measurability} shows that the map $Y\to \mc P(\R/\Z)$ taking $y$ to $\lambda_{n, y}$ is measurable for each $n\geq 1$.
	Therefore, the composite
	\begin{figure}[H]
		\centering
		\begin{tikzcd}
			&	(\eta_n:\ n\geq 1)\ar[r, mapsto] 	& (d(\eta_n, \lambda):\ n\geq 1)\\
			Y\ar[r] 	& \mc P(\R/\Z)^{\N}\ar[r] & 	\R^{\N}\\
			y\ar[r, mapsto]	& (\lambda_{n, y}:\ n\geq 1)\\
		\end{tikzcd}
	\end{figure}
	\noindent
	is measurable.
	If $S$ is the set of all the elements $(x_n:\ n\geq 1)$ in $\R^{\N}$ such that $x_n\to 0$, then by Lemma \ref{lemma:sequences converging to zero form a measurable set} we have $S$ is measurable.
	It is clear that $E$ is nothing but the preimage of $S$ under the composition above, and hence $E$ is measurable.

	Now we show that $K$ is measurable.
	For each $n, k\geq 1$, define
	\begin{equation}
		K_{n, k} = \set{y\in Y:\ \norm{\vp(y)- \vp(p\cdot y)}< 1/k \text{ for all } p\in [-n, n]^3}
	\end{equation}
	where $\norm{\cdot}$ denotes the distance to the nearest integer, the standard translation-invariant metric on $\R/\Z$.
	Each $K_{n,k}$ is measurable, being a finite intersection (over $p\in[-n,n]^3$) of preimages of an open set under the measurable maps $y\mapsto\vp(y)-\vp(p\cdot y)$.
	Now $K=\bigcap_{n, k\geq 1} K_{n, k}$ and hence $K$ is measurable.
\end{proof}


%
%
%

%
\subsection{Pure and Mixed Statistics}


Now suppose $(Y, \mc Y, \nu)$ is a probability measure space equipped with an ergodic $\Z^3$-action.
Let $\vp:Y\to \R/\Z$ be a measurable function.
For each $y\in Y$ we get a map $\vp_y:\Z^3\to \R/\Z$ defined as $\vp_y(v)=\vp(v\cdot y)$.

\begin{lemma}
	\label{lemma:pure statistics lemma}
	\textit{\textbf{Pure Statistics Lemma.}}
	Suppose the $3$-dimensional sequence $\vp_y:\Z^3\to \R/\Z$ is either a constant or equidistributed for $\nu$-a.e. $y\in Y$.
	Then $\vp_*\nu$ is either a Dirac measure or the Haar measure on $\R/\Z$.
\end{lemma}
\begin{proof}
	The set $E$ of points $y$ in $Y$ such that $\vp_y$ equidistributes form a measurable subset of $Y$.
	It is easy to check that this set is $\Z^3$-invariant.
	Thus this set either has full measure or has zero measure.
	Similarly, the set $K$ of points $y$ in $Y$ for which $\vp_y$ is constant either has full measure or zero measure.
	Now we have a full measure subset $S$ of $Y$ such that for each $y\in S$ we have
	\begin{equation}
		\lim_{n\to \infty}\int_{\R/\Z} F\ d\lrb{ \frac{1}{(2n+1)^3} \sum_{p\in [-n, n]^3} \delta_{\vp_y(p)}}
		=
		\lim_{n\to \infty}\frac{1}{(2n+1)^3} \sum_{p\in [-n, n]^3} F\circ \vp(p\cdot y)
	\end{equation}
	which by the Birkhoff ergodic theorem is equal to
	\begin{equation}
		\int_Y F\circ \vp\ d\nu = \int_{\R/\Z} F\ d(\vp_*\nu)
	\end{equation}
	for all $F\in C(\R/\Z)$.\footnote{First one may establish this for a countable dense subset of $C(\R/\Z)$ and then upgrade to all of $C(\R/\Z)$.}
	If $E$ has full measure, then, by the Riesz representation theorem we must have $\vp_*\nu$ is the Haar measure and if $K$ has full measure then we must have $\vp_*\nu$ is a Dirac measure.
	Since one of these two possibilities must occur, we are done.
\end{proof}

\begin{lemma}
	\label{lemma:mixed statistics lemma}
	\textit{\textbf{Mixed Statistics Lemma.}}
	Suppose for $\nu$-a.e. $y$ in $Y$ there is a finite index subgroup $\Gamma_y$ of $\Z^3$ such that the $3$-dimensional sequence obtained by restricting $\vp_y$ to any coset of $\Gamma_y$ is either a constant or equidistributed.
	Then there is a finite index subgroup $\Gamma$ of $\Z^3$ such that for any $\Gamma$-ergodic component $(E, \nu_E)$ of $Y$ we have $\vp_*\nu_E$ is either a Dirac measure or the Haar measure on $\R/\Z$.
\end{lemma}
\begin{proof}
	Let $\Gamma_1, \Gamma_2, \Gamma_3, \ldots$ be an enumeration of all the finite index subgroups of $\Z^3$.
	For each $i\geq 1$ let
	\begin{equation}
		Y_i =\bigcap_{g\in \Z^3} \lrp{
			\set{y\in Y|\ \vp_{g\cdot y}:\Gamma_i\to \R/\Z \text{ equidistributes}}
			\cup
			\set{y\in Y|\ \vp_{g\cdot y}:\Gamma_i\to \R/\Z \text{ is constant}}
		}
	\end{equation}
	By Lemma \ref{lemma:the equidistribution points form a measurable set} we know that each $Y_i$ is a measurable subset of $Y$.
	Note that each $Y_i$ is $\Z^3$-invariant.
	Define a map $\Psi:Y\to \mc S(\Z^3)$, where $\mc S(\Z^3)$ is the set of all the finite index subgroups of $\Z^3$, by writing
	\begin{equation}
		\Psi(y) = \Gamma_n
	\end{equation}
	where $n$ is the smallest positive integer $i$ such that on every coset of $\Gamma_i$ the sequence $\vp_y$ is either a constant or is equidistributed.
	Then note that
	\begin{equation}
		\Psi^{-1}(\Gamma_n) = Y_n\cap  \bigcap_{i=1}^{n-1} Y_i^c
	\end{equation}
	Therefore each fiber of $\Psi$ is measurable.\footnote{Thanks to Nishant Chandgotia for the argument showing the measurability of $\Psi$.}
	Since there are only countably many fibers of $\Psi$, there must exist a fiber with positive measure, whence by the ergodicity of the $\Z^3$ action we deduce that some fiber of $\Psi$ has full measure.
	Therefore, there is a finite index subgroup $\Gamma$ of $\Z^3$ such that for $\nu$-a.e. $y$ in $Y$ we have $\vp_y$ restricted to any coset of $\Gamma$ is either constant or equidistributed.
	Since $\Gamma$ acts ergodically on each $\Gamma$-ergodic component, the desired result follows from the \hyperref[lemma:pure statistics lemma]{Pure Statistics Lemma}.
\end{proof}

%
%
%

%
\subsection{Behaviour of Averages on Suitable Ergodic Components}

For a map $\vp:X\to \R$ we define a map $\bar \vp:X\to \R/\Z$ obtained by composing $\vp$ with the natural map $\R\to \R/\Z$.
There is a natural action of $\Z[u_1^\pm, u_2^\pm, u_3^\pm]$ on the set of maps taking $X$ into $\R/\Z$.\footnote{Recall that we write a typical element of $\Z[u_1^\pm, u_2^\pm, u_3^\pm]$ as $\sum_{v\in \Z^3} a_v U^v$.} 

For a subgroup $\Lambda$ of $\Z^3$ we write $\Lambda\otimes\R$ for the real subspace of $\R^3$ spanned by $\Lambda$,
\[
	\Lambda\otimes\R = \text{span}_\R(\Lambda) = \set{t_1\lambda_1+\cdots+t_m\lambda_m:\ m\geq 1,\ t_i\in\R,\ \lambda_i\in\Lambda}.
\]
Its dimension equals the rank of $\Lambda$.
In particular, each $\Lambda_i\otimes\R$ (with $\Lambda_i$ of rank $2$) is a plane through the origin.
The notation records the abstract tensor product $\Lambda\otimes_\Z\R$, a real vector space of dimension $\operatorname{rank}\Lambda$.
Since $\Lambda$ is a sublattice of $\Z^3$, the natural map $\Lambda\otimes_\Z\R\to\R^3$, $\lambda\otimes t\mapsto t\lambda$, is injective with image exactly this span, which is why we identify the two.
Note that $\Lambda\otimes\R$ is a real subspace, not a sublattice: it is not $\Z^3\cap(\Lambda\otimes\R)$, the latter being the (saturated) lattice of integer points of the span.

\begin{lemma}
	\label{lemma:unipotent two dimensions}
	\emph{(See \cite[Lemma 4.2]{bhattacharya_tilings})}
	Let $j\in \set{1, \ldots, n}$ be arbitrary.
	Then there is $h$ in $\Z^3$, lying in the real span of none of the $\Lambda_i$'s, such that
	\begin{equation}
		(U^h-1)^k\overline{f^{\Lambda_j}} = 0
	\end{equation}
	for some $k\geq 1$.
\end{lemma}
\begin{proof}
	The real spans $\Lambda_1\otimes\R, \ldots, \Lambda_n\otimes\R$ are finitely many two-dimensional subspaces of $\R^3$, each spanned by integer vectors and hence defined over $\Q$.
	We seek an integer vector lying in none of them, so we pass to the rational points.
	The intersections $(\Lambda_i\otimes\R)\cap\Q^3$ are proper $\Q$-subspaces of $\Q^3$.
	Since $\Q$ is infinite, $\Q^3$ is not the union of finitely many proper subspaces, so some $q\in\Q^3$ lies in none of them, and therefore in none of the $\Lambda_i\otimes\R$.
	Clearing denominators (which does not change which of these spans the vector belongs to) yields $h\in\Z^3$ lying in the real span of none of the $\Lambda_i$.
	Equivalently, $\Z h+\Lambda_i$ has rank $3$ for every $i$.
	
	Let $\Gamma_0$ be a finite index subgroup of $\Z^3$ such that $f_0$ is $\Gamma_0$-invariant (recall the definition of $f_0$ from Lemma \ref{lemma:very weakly periodic decomposition theorem}).
	Since $\Z h+\Lambda_j$ has rank $3$ it is of finite index, so $(\Z h + \Lambda_j)\cap \Gamma_0$ is of finite index.
	Fix a positive integer $l$ with $l\Z^3\subseteq (\Z h + \Lambda_j)\cap \Gamma_0$.
	For each $i\neq j$, the lattice $\Lambda_i$ is not contained in $\Lambda_j\otimes\R$ (because $\Lambda_i+\Lambda_j$ has rank $3$ by property (d) of Lemma \ref{lemma:choosing planes intersection and all that}), so we may choose $w_i\in \Lambda_i$ with $w_i\notin \Lambda_j\otimes\R$.
	Consider the Laurent polynomial
	\begin{equation}
		q(U)
		=
		(U^{lh}-1)\prod_{i:\ i\neq j} (U^{l w_i} - 1),
	\end{equation}
	where the leading factor $U^{lh}-1$ is present to handle $\overline{f_0}$ (and ensures $q\neq 1$ even when $n=1$, in which case the product is empty).

	We first show that $q$ annihilates $\overline{f^{\Lambda_j}}$.
	By Lemma \ref{lemma:very weakly periodic decomposition theorem} we have $f=f_0+\sum_{i=1}^n f^{\Lambda_i}$, and since $f=1_A$ is $\set{0,1}$-valued, $\overline f=0$.
	Reducing modulo $\Z$ therefore gives
	\begin{equation}
		\overline{f^{\Lambda_j}} = -\overline{f_0}-\sum_{i\neq j}\overline{f^{\Lambda_i}}.
	\end{equation}
	Now $lh\in l\Z^3\subseteq\Gamma_0$ and $f_0$ is $\Gamma_0$-invariant, so $(U^{lh}-1)\overline{f_0}=0$, and for each $i\neq j$ we have $lw_i\in\Lambda_i$ while $f^{\Lambda_i}$ is $\Lambda_i$-invariant, so $(U^{lw_i}-1)\overline{f^{\Lambda_i}}=0$.
	As the factors of $q$ commute, $q$ annihilates each of $\overline{f_0}$ and $\overline{f^{\Lambda_i}}$ for $i\neq j$, and hence, by the displayed decomposition, $q(U)\overline{f^{\Lambda_j}}=0$.

	On the other hand, since $l\Z^3\subseteq \Z h+\Lambda_j$, for each $i\neq j$ we may write
	\begin{equation}
		lw_i = a_ih + w_{ij}, \qquad a_i\in\Z,\ w_{ij}\in\Lambda_j,
	\end{equation}
	and, replacing $w_i$ by $-w_i$ if necessary, we may assume $a_i>0$.
	Note $a_i\neq 0$ since otherwise $lw_i=w_{ij}\in\Lambda_j\otimes\R$, contrary to the choice of $w_i$.
	Because $\overline{f^{\Lambda_j}}$ and all of its $\Z[u_1^\pm, u_2^\pm, u_3^\pm]$-translates are $\Lambda_j$-invariant and $w_{ij}\in\Lambda_j$, we have $U^{lw_i}\overline{f^{\Lambda_j}}=U^{a_ih}\overline{f^{\Lambda_j}}$.
	The leading factor $U^{lh}-1=(U^h)^l-1$ is already a power of $U^h$.
	Consequently
	\begin{equation}
		q(U)\overline{f^{\Lambda_j}}
		=
		(U^{lh}-1)\lrb{\prod_{i:\ i\neq j} (U^{a_i h} - 1)}\overline{f^{\Lambda_j}}.
	\end{equation}
	Combining the last two paragraphs, the operator $(U^{lh}-1)\prod_{i\neq j}(U^{a_ih}-1)$ annihilates $\overline{f^{\Lambda_j}}$.
	This is a product of unipotents in the single direction $h$, with positive exponents $l$ and $a_i$, so by Corollary \ref{corollary:annihilated by a polynomial unipotent} there are a positive integer $m$ and an integer $k\geq 1$ with $(U^{mh}-1)^k\overline{f^{\Lambda_j}}=0$.
	Since $h$ lies in the span of no $\Lambda_i$ and $m\neq 0$, the vector $mh$ also lies in the span of no $\Lambda_i$.
	Renaming $mh$ as $h$ completes the proof.
\end{proof}

\begin{lemma}
	\label{lemma:precursor to almost achieving weakly periodic decomposition}
	\emph{(See \cite[Lemma 4.3]{bhattacharya_tilings})}
	Let $i\in\set{1, \ldots, n}$ be given.
	Then there is a finite index subgroup $\Gamma_i$ of $\Z^3$ with the property that if $(E, \mu_E)$ is a $\Gamma_i$-ergodic component of $X$ then we have $\overline{f^{\Lambda_i}}_*\mu_E$ is either a Dirac measure or the Haar measure on $\R/\Z$.
\end{lemma}
\begin{proof}
	Write $\vp=\overline{f^{\Lambda_i}}$.
	By Lemma \ref{lemma:unipotent two dimensions} we can find $h\in \Z^3$, in the real span of no $\Lambda_{i'}$, and a positive integer $k$ such that $(U^h-1)^k\vp=0$.
	Also, $(U^g-1)\vp=0$ for all $g\in \Lambda_i$, since $f^{\Lambda_i}$ is $\Lambda_i$-invariant.
	Passing to the pointwise sequences, for $\mu$-a.e. $x\in X$ we have $(U^h-1)^k \vp_x = 0$ and $(U^g-1)\vp_x=0$ for all $g\in \Lambda_i$, where $\vp_x:\Z^3\to\R/\Z$ is given by $\vp_x(v)=\vp(v\cdot x)$.

	Fix such an $x$.
	Because $(U^g-1)\vp_x=0$ for every $g\in\Lambda_i$, the sequence $\vp_x$ is constant on each coset of $\Lambda_i$.
	Since $h$ lies outside $\Lambda_i\otimes\R$, the subgroup $\Lambda_i+\Z h$ has rank $3$, hence finite index, so there are only finitely many cosets of $\Lambda_i+\Z h$ in $\Z^3$.
	On any one of these cosets the only variation of $\vp_x$ is in the $h$-direction, where $(U^h-1)^k\vp_x=0$ exhibits $t\mapsto\vp_x(\,\cdot\,+th)$ as a (one-dimensional) polynomial sequence by Lemma \ref{lemma:unipotent polynomial on cosets}.
	By Weyl's equidistribution theorem each such sequence either equidistributes or is periodic.
	Let $d\geq 1$ be a common multiple of the (finitely many) periods arising from the periodic cosets, and put $\Gamma_x = \Lambda_i+\Z(dh)$, a finite index subgroup of $\Z^3$.
	On every coset of $\Gamma_x$ the sequence $\vp_x$ is then constant (if the underlying $h$-direction sequence was periodic, $d$ being a multiple of its period) or equidistributed (if it was equidistributed).
	Now by the \hyperref[lemma:mixed statistics lemma]{Mixed Statistics Lemma} we can find a finite index subgroup $\Gamma_i$ of $\Z^3$ such that $\vp_*\mu_E$ is either a Dirac measure or the Haar measure for any $\Gamma_i$-ergodic component $(E, \mu_E)$ of $X$.
\end{proof}

\begin{remark}
	It can be easily argued that the conclusion of the above lemma does not change if we pass to any finite index subgroup of $\Gamma_i$.
	More precisely, the above lemma can be made more nuanced by saying that there is a finite index subgroup $\Gamma_i$ 
	such that whenever $\Gamma$ is any finite index subgroup of $\Gamma_i$ we have $\overline{f^{\Lambda_i}}_*\mu_E$ is either the Haar measure on $\R/\Z$ or a Dirac measure on $\R/\Z$ for any $\Gamma$-ergodic component $E$ of $X$.
	We will make use of this observation in what follows.
\end{remark}

The strategy to show that $A$ is $1$-weakly periodic is to show that there is a finite index subgroup $\Gamma$ of $\Z^3$ such that the intersection of $A$ with any $\Gamma$-ergodic component of $X$ is $1$-weakly periodic.
Our next lemma almost achieves this.
Since the conclusion of the above lemma does not change when $\Gamma_i$ is replaced by any finite index subgroup of $\Gamma_i$, this allows us to prove the following.

\begin{lemma}
	\label{lemma:almost achieving weakly periodic decomposition}
	\emph{(See \cite[Theorem 4.4]{bhattacharya_tilings})}
	There is a finite index subgroup $\Gamma$ of $\Z^3$ such that for each $\Gamma$-ergodic component $(E, \mu_E)$ of $X$, we have either $\mu_E(A\cap E) = 1/2$ or $A\cap E$ is $2$-periodic (or both).
\end{lemma}
\begin{proof}
	Let $\Gamma_0$ be a finite index subgroup of $\Z^3$ such that $f_0$ is $\Gamma_0$-invariant.
	By Lemma \ref{lemma:precursor to almost achieving weakly periodic decomposition} we know that for each $i\in\set{1, \ldots, n}$ we have a finite index subgroup $\Gamma_i$ of $\Z^3$ 
	such that for each $\Gamma_i$-ergodic component $E$ of $X$ we have $\overline{f^{\Lambda_i}}_*\mu_E$ is either a Dirac measure or the Haar measure on $\R/\Z$.
	Define $\Gamma=\Gamma_0\cap\bigcap_{i=1}^n\Gamma_i$.
	We will show that $\Gamma$ satisfies the requirement of the lemma.
	Fix a $\Gamma$-ergodic component $E$ of $X$.
	There are two cases.

	\emph{Case 1: There is $i\in \set{1, \ldots, n}$ such that $\overline{f^{\Lambda_i}}_*\mu_E$ is the Haar measure $\lambda$ on $\R/\Z$.}
	We will show that $\mu_E(A\cap E)=1/2$.
	Note that since $f$, and hence $f^{\Lambda_i}$, is valued in the unit interval $I$, the fact that $\overline{f^{\Lambda_i}}_*\mu_E$ is the Haar measure on $\R/\Z$ implies that $f^{\Lambda_i}_*\mu_E$ is the Haar measure $\lambda_I$ on the unit interval $I$.
	Let $l\geq 1$ be such that $l\Z^3\subseteq \Gamma$.
	By Lemma \ref{lemma:a special periodicity lemma} we have $f^{l \Lambda_i} - f^{\Lambda_i}$ is $3$-periodic.
	So there is a finite index subgroup $\Gamma'$ of $\Z^3$ such that $f^{l \Lambda_i}-f^{\Lambda_i}$ is invariant under $\Gamma'$.
	Let $\Theta=\Gamma\cap \Gamma'$ and $(E_1, \mu_1), \ldots, (E_k, \mu_k)$ be all the $\Theta$-ergodic components of $X$ which are contained in $E$.
	For each $j\in \set{1, \ldots, k}$ let $f^{l\Lambda_i}-f^{\Lambda_i}\equiv c_j$ on $E_j$, where $c_j$ is some real number.
	By the remark following Lemma \ref{lemma:precursor to almost achieving weakly periodic decomposition} (applied to the finite index subgroup $\Theta\subseteq\Gamma_i$), each $\overline{f^{\Lambda_i}}_*\mu_j$ is either the Haar measure or a Dirac measure on $\R/\Z$.
	Now $\mu_E$ is the convex combination $\sum_{j=1}^k \mu_E(E_j)\,\mu_j$, so $\overline{f^{\Lambda_i}}_*\mu_E = \sum_j \mu_E(E_j)\,\overline{f^{\Lambda_i}}_*\mu_j$.
	Since this equals the non-atomic Haar measure, no $\mu_j$ of positive weight can contribute an atom, so $\overline{f^{\Lambda_i}}_*\mu_j$ is the Haar measure for every $j$, and hence $f^{\Lambda_i}_*\mu_j$ is the Haar measure on the unit interval.
	Therefore $f^{l\Lambda_i}_*\mu_j$ is the Haar measure on the interval $[c_j, c_j+1]$.
	But $f^{l\Lambda_i}$ is valued in $I=[0,1]$, and therefore $c_j$ must be $0$.
	Since this holds for every $j$, we have $f^{l\Lambda_i}=f^{\Lambda_i}$ on $E$, and hence $f^{l\Lambda_i}_*\mu_E$ is also the Haar measure on the unit interval.
	Now 
	\begin{equation}
		\mu_E(A\cap E) = \int_E 1_A\ d\mu_E = \int_E f\ d\mu_E = \int_E f^{l \Lambda_i}\ d\mu_E
	\end{equation}
	where the last equality is because of the ergodic theorem coupled with the fact that $E$ is $\Gamma$-invariant, and hence $l\Lambda_i$-invariant.
	But since $f^{l\Lambda_i}_*\mu_E$ is the Haar measure on $I$, we see that $\mu_E(A\cap E)= 1/2$.

	\emph{Case 2: There is no $i\in \set{1, \ldots, n}$ such that $\overline{f^{\Lambda_i}}_*\mu_E$ is the Haar measure on $\R/\Z$.}
	Thus $\overline{f^{\Lambda_i}}_*\mu_E$ is a Dirac measure for each $i$.
	We divide this into two subcases.

	\emph{Subcase 2.1: There is $i\in \set{1, \ldots, n}$ such that $\overline{f^{\Lambda_i}}_*\mu_E$ is the Dirac measure at $0$.}
	We will show that $A\cap E$ is $2$-periodic.
	Since $\overline{f^{\Lambda_i}}_*\mu_E=\delta_0$, we have $f^{\Lambda_i}\in\Z$ $\mu_E$-almost everywhere on $E$.
	Combined with $f^{\Lambda_i}\in I=[0,1]$, this means $f^{\Lambda_i}|_E$ takes values in the two-element set $\set{0, 1}$.
	Choose $l\geq 1$ such that $l\Lambda_i\subseteq \Gamma$. 
	Since $l\Lambda_i\subseteq\Lambda_i$, the $\Lambda_i$-invariant $\sigma$-algebra is contained in the $l\Lambda_i$-invariant one, so by the tower property $f^{\Lambda_i}$ is the conditional expectation of $f^{l\Lambda_i}$ given the $\Lambda_i$-invariant $\sigma$-algebra.
	On the $\Lambda_i$-invariant set $\set{f^{\Lambda_i}=1}$ we then have $\int f^{l\Lambda_i} = \int f^{\Lambda_i} = \mu_E(\set{f^{\Lambda_i}=1})$ while $f^{l\Lambda_i}\leq 1$, forcing $f^{l\Lambda_i}=1$ there.
	Similarly, $f^{l\Lambda_i}=0$ on $\set{f^{\Lambda_i}=0}$ since $f^{l\Lambda_i}\geq 0$.
	Thus $f^{l\Lambda_i}$ coincides with $f^{\Lambda_i}$ on $E$.
	Now let
	\begin{equation}
		S=\set{x\in E:\ f^{\Lambda_i}(x) = 1} = \set{x\in E:\ f^{l\Lambda_i}(x) = 1}
	\end{equation}
	Then
	\begin{equation}
		\mu_E(S)
		=
		\int_E 1_S\ d\mu_E
		=
		\int_E f^{l\Lambda_i} \ d\mu_E \reason{(*)}{=}
		\int_E f\ d\mu_E
		= \mu_E(A)
	\end{equation}
	where $(*)$ is because $E$ is $\Gamma$-invariant and hence is $l\Lambda_i$-invariant.
	We also have
	\begin{equation}
		\int_E f^{l\Lambda_i} \ d\mu_E = \int_S f^{l\Lambda_i}\ d\mu_E = \int_S f\ d\mu_E = \mu_E(A\cap S)
	\end{equation}
	because $f^{l\Lambda_i}|_E$ is the same as $1_S$ and $S$ is $l\Lambda_i$-invariant.
	We deduce that $\mu_E(A) = \mu_E(S) = \mu_E(A\cap S)$ and therefore $A\cap E = S$.
	Since $S$ being $l\Lambda_i$-invariant is $2$-periodic, we have $A\cap E$ is $2$-periodic.

	\emph{Subcase 2.2: There is no $i\in \set{1, \ldots, n}$ such that $\overline{f^{\Lambda_i}}_*\mu_E$ is the Dirac measure at $0$.}
	In this case, for each $i$ the measure $\overline{f^{\Lambda_i}}_*\mu_E$ is the Dirac measure at some point $c_i\neq 0$ of $\R/\Z$.
	Since $f^{\Lambda_i}$ is valued in $I=[0,1]$ and the only point of $[0,1]$ reducing to $c_i\neq 0$ modulo $\Z$ is $c_i$ itself, it follows that $f^{\Lambda_i}=c_i$ is constant on $E$.
	As $\Gamma\subseteq\Gamma_0$, the function $f_0$ is $\Gamma$-invariant; hence, by $f = f_0+ \sum_{i=1}^n f^{\Lambda_i}$, the restriction $f|_E$ is $\Gamma$-invariant and so $3$-periodic, and thus $A\cap E$ is $3$-periodic (which in particular implies that $A\cap E$ is $2$-periodic).
	This completes the proof.
\end{proof}

%
%
%

%

\subsection{Weakly Periodic Decomposition}

\begin{lemma}
	\label{lemma:complement of weakly periodic is weakly periodic}
	\emph{(See \cite[Lemma 5.1]{bhattacharya_tilings})}
	Let $(Y, \mc Y, \nu)$ be a probability space equipped with an ergodic $\Z^3$-action.
	If $E\subseteq Y$ is $3$-periodic and $B\subseteq E$ is $2$-weakly periodic, then $E\setminus B$ is also $2$-weakly periodic.
\end{lemma}
\begin{proof}
	Let $B=B_1\sqcup \cdots \sqcup B_k$ be a $2$-weakly periodic decomposition of $B$ and let $\Theta_1, \ldots, \Theta_k$ be rank-$2$ subgroups of $\Z^3$ such that each $B_i$ is $\Theta_i$-invariant.
	We may assume that $k$ is the smallest positive integer satisfying the above.
	Thus $\Theta_i+\Theta_j$ is a finite index subgroup of $\Z^3$ whenever $i\neq j$.
	Let $\Gamma_0$ be a finite index subgroup of $\Z^3$ such that $E$ is $\Gamma_0$-invariant.
	Define $\Gamma_{ij} = \Gamma_0\cap \Theta_i+\Gamma_0\cap \Theta_j$ whenever $i\neq j$, and note that each $\Gamma_{ij}$ is a finite index subgroup of $\Z^3$.
	Define $\Gamma= \bigcap_{i\neq j}\Gamma_{ij}$.
	Note that $E$ is partitioned by the $\Gamma$-ergodic components contained in $E$ since $E$ is $\Gamma$-invariant.

	We will show that for every $\Gamma$-ergodic component $Q$ of $(X, \mu)$, the set $Q\cap (E\setminus B)$ is $2$-periodic.
	Indeed, fix a $\Gamma$-ergodic component $Q$ such that $Q\cap B$ has positive measure.
	Then $Q$ must be contained in $E$.
	Suppose there exist $i, j\in \set{1, \ldots, k}$, $i\neq j$, such that $Q\cap B_i$ and $Q\cap B_j$ both have positive measure.
	Let $S=\bigcup_{g\in \Gamma_{ij}} g(Q\cap B_j)$.
	Then $S$ is $\Gamma_{ij}$-invariant and hence $\Gamma$-invariant.
	Thus $S$ must contain $Q$.
	But note that if $g\in \Gamma_{ij}$, then $g$ can be written as $g=g_i+g_j$ for some $g_i\in \Gamma_0\cap \Theta_i$ and $g_j\in \Gamma_0\cap \Theta_j$.
	Therefore
	\begin{equation}
		g(Q\cap B_j) = (g_i+g_j)(Q\cap B_j) = g_i(g_j(Q\cap B_j)) \subseteq g_i(g_jB_j) \subseteq g_iB_j
	\end{equation}
	Now since $B_i$ and $E$ are both $g_i$-invariant, and $B_j\subseteq E$ is disjoint with $B_i$, we must have $g_iB_j\subseteq E\setminus B_i$.
	This means that the set $S$ is disjoint with $B_i$, contradicting the fact that $S$ contains $Q$.

	So we see that each $\Gamma$-ergodic component can intersect at most one of the $B_i$'s in a positive measure set.
	Finally, if $Q$ is a $\Gamma$-ergodic component intersecting $B_i$ in a positive measure set, and $l$ is a positive integer such that $l\Theta_i\subseteq \Gamma$, then $Q$ and $B_i$ are both $l\Theta_i$-invariant.
	Thus	$Q\setminus B = Q\setminus B_i$ is also $l\Theta_i$-invariant, and is hence $2$-periodic.
\end{proof}

By translating $F$ if necessary, we can manage that $0\in F$ along with the property that the sum of any set of nonzero vectors in $F$ is nonzero.
With this in mind we now finish the proof of the $1$-weak periodicity of $A$, thereby establishing the existence of a $1$-weakly periodic $F$-tiling.

\begin{lemma}
	Let $\Gamma$ be as in Lemma \ref{lemma:almost achieving weakly periodic decomposition}.
	If $A\cap E$ is not $2$-weakly periodic for some $\Gamma$-ergodic component $(E, \mu_E)$ of $\mu$, then $A\cap E$ is $1$-periodic.
\end{lemma}
\begin{proof}
	(See \cite[Proof of Lemma 1.3, pg 13]{bhattacharya_tilings})
	For any vector $v\in \Z^3$, we will write $\bar v$ to denote $-v$.
	Let $E$ be a $\Gamma$-ergodic component of $X$ such that $A\cap E$ is not $2$-weakly periodic.
	Then $\mu_E(A\cap E) = 1/2$.
	Now we have
	\begin{equation}
		E\setminus (A\cap E) = \bigsqcup_{g\neq 0, g\in F} (gA)\cap E
	\end{equation}
	since $\set{gA:\ g\in F}$ is a partition of $X$.
	If each $(gA)\cap E$ is $2$-weakly periodic then so is $E\setminus (A\cap E)$, and hence, by Lemma \ref{lemma:complement of weakly periodic is weakly periodic}, $A\cap E$ is also $2$-weakly periodic, contrary to our assumption.
	Thus there is $b_1\in F\setminus \set{0}$ such that $(b_1 A)\cap E$ is not $2$-weakly periodic.
	This implies that $A\cap \bar b_1E$ is also not $2$-weakly periodic, and hence 
	$$
	1/2
	=
	\mu_{\bar b_1 E}(A\cap \bar b_1E)
	=
	\mu_E((b_1A) \cap E)
	$$
	So we have $E=(A\cap E)\sqcup ((b_1A)\cap E)$, which implies that
	$$
	\bar b_1 E = (A\cap \bar b_1E) \sqcup (\bar b_1(A\cap E))
	$$
	As noted earlier, we have $A\cap \bar b_1E$ is not $2$-weakly periodic.
	By the same argument as above applied to $A\cap \bar b_1E$, we deduce that there is $b_2\in F\setminus\set{0}$ such that $\bar b_1E = (A\cap \bar b_1 E) \sqcup ((b_2 A) \cap (\bar b_1 E))$, which shows that
	\begin{equation}
		(b_2A)\cap \bar b_1E = \bar b_1(A\cap E)
	\end{equation}
	and hence $A\cap (\bar b_1+\bar b_2)E = (\bar b_1+\bar b_2) (A\cap E)$.
	Thus, continuing this way, for each $n\geq 1$ we can find $b_1, \ldots, b_{2n}\in F\setminus \set{0}$ such that
	\begin{equation}
		A\cap (\bar b_1+ \cdots + \bar b_{2n})E = (\bar b_1+ \cdots + \bar b_{2n})(A\cap E)
	\end{equation}
	Since there are only finitely many $\Gamma$-ergodic components, we see that there must exist natural numbers $m$ and $n$ with $m<n$ such that
	\begin{equation}
		(\bar b_1+ \cdots +\bar b_{2m}) E = (\bar b_1+ \cdots +\bar b_{2n}) E
	\end{equation}
	and hence
	\begin{equation}
		(\bar b_1+ \cdots +\bar b_{2m}) (A\cap E) = (\bar b_1+ \cdots +\bar b_{2n}) (A\cap E)
	\end{equation}
	giving
	\begin{equation}
		(\bar b_{2m+1} + \cdots +\bar b_{2n}) (A\cap E) = A\cap E
	\end{equation}
	and hence, since $\bar b_{2m+1} + \cdots + \bar b_{2n}$ is nonzero, $A\cap E$ is $1$-periodic.
	This proves the lemma.
\end{proof}

\begin{corollary}
	Then $A$ is $1$-weakly periodic.
\end{corollary}
\begin{proof}
	Let $\Gamma$ be as in Lemma \ref{lemma:almost achieving weakly periodic decomposition} and apply Lemma \ref{lemma:complement of weakly periodic is weakly periodic}.
\end{proof}

\bibliographystyle{abbrv}
\bibliography{UBOP.bib}

\begin{thebibliography}{10}

\bibitem{beauquier_nivat_1991}
D.~Beauquier and M.~Nivat.
\newblock On translating one polyomino to tile the plane.
\newblock {\em Discrete Comput. Geom.}, 6(6):575--592, 1991.

\bibitem{bhattacharya_tilings}
S.~Bhattacharya.
\newblock Periodicity and decidability of tilings.
\newblock {\em arXiv 1602.05738v1}, 2016.

\bibitem{coven_meyerowitz_tiling_integers}
E.~M. Coven and A.~Meyerowitz.
\newblock Tiling the integers with translates of one finite set.
\newblock {\em J. Algebra}, 212(1):161--174, 1999.

\bibitem{einsiedler_ward_ergodic_theory}
M.~Einsiedler and T.~Ward.
\newblock {\em Ergodic theory with a view towards number theory}, volume 259 of
  {\em Graduate Texts in Mathematics}.
\newblock Springer-Verlag London, Ltd., London, 2011.

\bibitem{greenfeld_tao_2020}
R.~Greenfeld and T.~Tao.
\newblock The structure of translational tilings in $\mathbb{Z}^d$.
\newblock {\em Discrete Anal.}, pages Paper No. 16, 28, 2021.

\bibitem{greenfeld_tao_undecidable_trans}
R.~Greenfeld and T.~Tao.
\newblock Undecidable translational tilings with only two tiles, or one
  nonabelian tile.
\newblock {\em Discrete Comput. Geom.}, 70(4):1652--1706, 2023.

\bibitem{greenfel_tao_counterexample_ptc}
R.~Greenfeld and T.~Tao.
\newblock A counterexample to the periodic tiling conjecture.
\newblock {\em Ann. of Math. (2)}, 200(1):301--363, 2024.

\bibitem{horak_kim_algebraic_method}
P.~Horak and D.~Kim.
\newblock Algebraic method in tilings.
\newblock {\em arXiv 1603.00051v1}, 2016.

\bibitem{kari_szabados_alg_geom}
J.~Kari and M.~Szabados.
\newblock An algebraic geometric approach to {N}ivat's conjecture.
\newblock {\em Inform. and Comput.}, 271:104481, 25, 2020.

\bibitem{lagarias_wang_tiling}
J.~C. Lagarias and Y.~Wang.
\newblock Tiling the line with translates of one tile.
\newblock {\em Invent. Math.}, 124(1-3):341--365, 1996.

\bibitem{szegedy_algorithms_to_tile}
M.~Szegedy.
\newblock Algorithms to tile the infinite grid with finite clusters.
\newblock {\em Foundations of Computer Science}, 1998.

\bibitem{tijdeman_decomposition_of}
R.~Tijdeman.
\newblock Decomposition of the integers as a direct sum of two subsets.
\newblock 215:261--276, 1995.

\end{thebibliography}

\end{document}